\documentclass[EJP]{ejpecp}

\usepackage[T1]{fontenc}
\usepackage[utf8]{inputenc}

\usepackage{enumerate}
\usepackage{caption}
\usepackage{color}

\SHORTTITLE{Transience of conditioned walks on the plane}

\TITLE{Transience of conditioned walks on the plane:
encounters and speed of escape}

\AUTHORS{%
Serguei~Popov\footnote{CMUP -- Centro de Matem\'atica da Univ. do Porto,
rua do Campo Alegre, 687, 4169-007 Porto, Portugal.
\BEMAIL{serguei.popov@fc.up.pt}}
\footnote{Department of Statistics, Institute of Mathematics,
Statistics and Scientific Computation, University of Campinas --
UNICAMP, rua S\'ergio Buarque de Holanda 651,
13083--859, Campinas SP, Brazil.
\BEMAIL{danielungaretti@gmail.com}}
\and%
Leonardo T.~Rolla\footnote{
Argentina National Research Council at the University of Buenos Aires;
NYU-ECNU Institute of Mathematical Sciences at NYU Shanghai.
\EMAIL{leorolla@dm.uba.ar}}
\and Daniel~Ungaretti\footnotemark[2]
}

\KEYWORDS{simple random walk; conditioning;
transience; Doob's $h$-transform}

\AMSSUBJ{60J10} \AMSSUBJSECONDARY{60G50; 82C41}

\SUBMITTED{November 14, 2019}
\ACCEPTED{April 12, 2020}

\ARXIVID{1910.13517v1}

\VOLUME{25}
\YEAR{2020}
\PAPERNUM{52}
\DOI{10.1214/20-EJP458}

\ABSTRACT{
We consider the two-dimensional simple random walk
conditioned on never hitting the origin, which is,
formally speaking,
the Doob's $h$-transform of the simple random walk
with respect to the potential kernel.
We then study the behavior of the future
minimum distance of the walk to the origin, and also
prove that two independent copies
of the conditioned walk,
although both transient, will nevertheless meet
infinitely many times a.s.
}

\renewcommand{\subset}{\subseteq}

\newcommand{\htau}{\widehat{\tau}}

\newcommand{\hS}{\smash{\widehat{S}}}
\newcommand{\hT}{\smash{\widehat{T}}}
\newcommand{\ind}{\mathbb{1}}

\begin{document}


\section{Introduction and main results}
\label{sec:introduction_and_results}

This paper is about the simple random walk (SRW) on the two-dimensional
lattice conditioned on not hitting the origin; formally speaking, it is
the Doob's $h$-transform of two-dimensional SRW with respect to the potential
kernel. This random walk (denoted by~$\smash {\widehat {S}}$) is the main
building block of two-dimensional random interlacements introduced in~\cite{comets2016two}
and further studied in~\cite{comets2017vacant} and~\cite{Rod19} (there
is also a continuous version of this process~\cite{comets2019two}). The
two-dimensional random interlacement process is related to (now classical)
random interlacements in dimensions $d\geqslant 3$ (cf.~\cite{CerTei12,DreRatSap14,sznitman2010vacant})
which arise as the limit of the local picture produced by the SRW on the
torus $\mathbb {Z}^{d}_{n}=\mathbb {Z}^{d}/n\mathbb {Z}^{d}$ up to time~$un^{d}$,
and can be described as a canonical Poisson soup of SRW trajectories. The
same construction is impossible in dimension two, simply due to the fact
that the two-dimensional SRW is recurrent. However, if we
\emph{condition} the SRW on the two-dimensional torus~$\mathbb {Z}^{2}_{n}$
to avoid a fixed site at an appropriate time scale, then the local picture
around this site has a limit again. This limit is the two-dimensional random
interlacement process, which is made of conditioned (on not hitting the
origin) SRW trajectories, defined through the canonical construction of
random interlacements on transient weighted graphs of~\cite{teixeira2009interlacement}.
A similar construction can be done also in one dimension~\cite{camargo2016one}.

The conditioned walk~$\smash {\widehat {S}}$ then became an interesting
object on its own. Some of its (sometimes surprising) properties shown
in~\cite{comets2016two,gantert2018range} include
\begin{itemize}
\item $\smash {\widehat {S}}$ is transient; however,
\[
\lim _{y\to \infty } \mathbb {P}_{x_{0}}[\smash {\widehat {S}}
\text{ ever hits }y] = \frac{1}{2}
\]
for any $x_{0}\neq 0$;
\item \emph{any} infinite set is recurrent for~$\smash {\widehat {S}}$;
\item if~$A$ is a ``nice'' large set (e.g., a large disk or square or segment),
then the proportion of sites of~$A$ which are ever visited by~$
\smash {\widehat {S}}$ is a random variable approximately distributed as
Unif$[0,1]$.
\end{itemize}

In this paper, our aim is to further examine the properties of the conditioned
walk~$\smash {\widehat {S}}$. This work may be better described as ``further
quantifying the transience of the conditioned walk.'' Namely, we first
consider the process~$M_{n}$ which is the \emph{future} minimum distance
of the walk to the origin, and show in Theorem~\ref{teo:speed} that it
has quite an ``erratic'' behavior. Further, we will prove in Theorem~\ref{teo:encounters}
that two independent copies of~$\smash {\widehat {S}}$, although both transient,
will nevertheless meet infinitely many times a.s.

Now, let us introduce some notation. For $x, y \in \mathbb {Z}^{2}$ we
say that~$x$ and~$y$ are nearest neighbors if they are at Euclidean distance
$1$, and denote that by~$x \sim y$. Classical simple random walk on the
plane is the random walk on~$\mathbb {Z}^{2}$ that jumps to nearest neighbors
with uniform probabilities. We denote it by
$(S_{n};\; n \geqslant 0)$. For $A \subset \mathbb {Z}^{2}$ we define the
stopping times
\begin{equation}
\label{eq:hitting_times}
\tau (A) := \inf \{ n \geqslant 0;\; S_{n} \in A\} \quad \text{and}
\quad \tau _{+}(A) := \inf \{ n \geqslant 1;\; S_{n} \in A\}.
\end{equation}

When $A = \{x\}$ is a singleton we write simply $\tau (x)$. When studying
properties of $S_{n}$, a fundamental quantity to understand its behavior
is the \emph{potential kernel}
\begin{equation*}
a(x) := \sum _{k=0}^{\infty } \bigl( \mathbb {P}_{0} [S_{k} = 0] -
\mathbb {P}_{x} [S_{k} = 0] \bigr).
\end{equation*}
It is known that $a(0) = 0$, $a(x) = 1$ for any of the $4$ neighbors of
the origin and $a(x) > 0$ for every other $x \in \mathbb {Z}^{2}$. Moreover,
the potential kernel is a harmonic function outside the origin:
\begin{equation*}
a(x) = \frac{1}{4} \sum _{y;\; y \sim x} a(y), \quad
\text{for $x \neq 0$.}
\end{equation*}
This property is the basis of many estimates regarding SRW on the plane,
since it renders $a(S_{n \wedge \tau (0)})$ a martingale. An asymptotic
expression for the potential kernel is given by
\begin{equation}
a(x) = \frac{2}{\pi } \ln |x| + \frac{2 \gamma + 3 \ln 2}{\pi } + O(|x|^{-2}),
\end{equation}
in which $\gamma $ is the Euler-Mascheroni constant. These properties and
more can be found in~\cite{lawler2010random}. Using the potential kernel~$a$,
we define a random walk
$(\smash {\widehat {S}}_{n};\; n \geqslant 0)$ on $\mathbb {Z}^{2}$ that
is closely related to SRW. For presentation's sake, we show their transition
probabilities side by side:
\begin{equation*}
P_{x,y} =
\begin{cases}
\frac{1}{4} &\text{if $y \sim x$},
\\
0 &\text{otherwise},
\end{cases}
\quad \text{and} \quad \widehat{P}_{x,y} =
\begin{cases}
\frac{a(y)}{a(x)} \times \frac{1}{4} &\text{if $y \sim x$},
\\
0 &\text{otherwise}.
\end{cases}
\end{equation*}

Random walk $\smash {\widehat {S}}$ is the Doob $h$-transform of SRW conditioned
on not hitting the origin. It is a key object in defining random interlacements
on $\mathbb {Z}^{2}$, cf.~\cite{comets2016two}. In what follows, we consider
that all random walks are built in a common probability space with probability
and expectation denoted by $\mathbb {P}$ and $\mathbb {E}$, respectively.
We use notation~$\mathbb {P}_{x}$ to explicit the starting point of the
walk and~$\mathbb {P}_{x_{1}, x_{2}}$ when there are two walks involved.

If $r > 0$ and $x \in \mathbb {Z}^{2}$, we denote by
$B(x, r) = \{y \in \mathbb {Z}^{2};\; |y-x| \leqslant r\}$ the discrete
Euclidean disk and define $B(r) := B(0, r)$. Also, we define the internal
boundary of a set $A \in \mathbb {Z}^{2}$ as
\begin{equation*}
\partial A := \{x \in A;\;
\text{there is $y \in \mathbb {Z}^{2}\setminus A$ such that
$x \sim y$}\}.
\end{equation*}
Sometimes, instead of $\mathbb {P}_{x}$ with $x \in \mathbb {Z}^{2}$ we
write $\mathbb {P}_{r}$ for some $r > 0$; this means that the expression
works for any $x \in \partial B(r)$. We also define for $r>0$
\begin{equation*}
a(r) := \frac{2}{\pi } \ln r + \frac{2 \gamma + 3 \ln 2}{\pi }.
\end{equation*}
Moreover, we denote by $\htau (A)$ and $\htau _{+}(A)$ the stopping times
for the $\smash {\widehat {S}}$-walk that are analogous to~\eqref{eq:hitting_times}
and define $\htau (r) := \htau (\partial B(r))$ for $r > 0$.

The next proposition is adapted from~\cite[Proposition~1.1]{gantert2018range}
and presents some of the basic properties of an
$\smash {\widehat {S}}$-walk that were proved in~\cite{comets2016two}.
We also add an explicit expression for
$\smash {\widehat {G}}(x, y) = \mathbb {E}_{x} [\# \
\text{visits to $y$}]$, the Green function of an
$\smash {\widehat {S}}$-walk at item \textit{(vi)}.
\begin{proposition}
\label{prop:properties_hS}
The following statements hold:
\begin{itemize}
\item[(i)] The walk~$\smash {\widehat {S}}$ is reversible, with the reversible
measure~$\mu _{x}:=a^{2}(x)$.
\item[(ii)] In fact, it can be represented as a random walk on the two-dimensional
lattice with conductances
$\big(a(x)a(y), x,y\in \mathbb {Z}^{2}, x\sim y\big)$.
\item[(iii)] Let~$\mathcal{N}$ be the set of the four neighbors of the
origin. Then the process
$1/a(\smash {\widehat {S}}_{k\wedge \htau (\mathcal{N})})$ is a martingale.
\item[(iv)] The walk $\smash {\widehat {S}}$ is transient.
\item[(v)] For all $x\neq y$ with $x,y\neq 0$ we have
\begin{align*}
\mathbb {P}_{x}\big[\htau _{+}(x)<\infty \big] &= 1 - \frac{1}{2a(x)},
\\
\mathbb {P}_{x}\big[\htau (y)<\infty \big] &=
\frac{a(x)+a(y)-a(x-y)}{2a(x)}.
\end{align*}
\item[(vi)] The Green function of an $\smash {\widehat {S}}$-walk is given
by
\begin{equation}
\label{eq:conditional_green_function}
\smash {\widehat {G}}(x, y) = \frac{a(y)}{a(x)} \cdot (a(x) + a(y) - a(x
- y)).
\end{equation}
\end{itemize}
\end{proposition}

Item \textit{(vi)} is a consequence of \textit{(v)}. A straightforward derivation
of $\textit{(vi)}$ without the aid of $\textit{(v)}$ and further potential-theoretic
results about the $\smash {\widehat {S}}$-walk can be found in~\cite{popov2019conditioned}.

Reference~\cite{gantert2018range} is devoted to proving properties related
to how an $\smash {\widehat {S}}$-walk intersects some sets. This work
advances in the direction of comprehending trajectories of
$\smash {\widehat {S}}$, but focuses on almost sure properties related
to its speed of transience and the relationship between two independent
copies of $\smash {\widehat {S}}$.

Let us state the main results from the paper. By Proposition~\ref{prop:properties_hS},
we know that $\smash {\widehat {S}}_{n}$ is transient. Our first result
is a quantitative assessment of how fast this transience happens. Let us
define
\begin{equation*}
M_{n} := \min _{m \geqslant n} |\smash {\widehat {S}}_{m}| \quad
\text{and} \quad T_{u} := \sup \{n \geqslant 0;\; |
\smash {\widehat {S}}_{n}| \leqslant u\}.
\end{equation*}

There is a simple relation between these quantities, which can be seen
by
\begin{equation}
\label{eq:relation_Mn_Tu}
\{ T_{u} \geqslant n\} = \big\{\text{there exists }m \geqslant n
\text{ such that } |\smash {\widehat {S}}_{m}| \leqslant u \big\} = \{M_{n}
\leqslant u\}.
\end{equation}

We develop some almost sure asymptotic properties of $M_{n}$. Each of them
can be translated into a property for $T_{u}$ via~\eqref{eq:relation_Mn_Tu}.
\begin{theorem}
\label{teo:speed}
For every $0 < \delta < \frac{1}{2}$ we have, almost surely,
\begin{equation}
\label{eq:teo_speed_io}
M_{n} \leqslant n^{\delta }\ \text{i.o.} \quad \text{but} \quad M_{n}
\geqslant \tfrac{\sqrt{n}}{\ln ^{\delta } n}\ \text{i.o.}
\end{equation}
and, on the other hand,
\begin{equation}
\label{eq:teo_speed_eventually}
e^{\ln ^{1-\delta } n} \leqslant M_{n} \leqslant \sqrt{(e+\delta )n (
\ln \ln n)} \quad \text{ eventually}.
\end{equation}
\end{theorem}
Quantity $M_{n}$ is non-trivial only for transient walks since for a recurrent
walk, e.g. SRW on $\mathbb {Z}^{2}$, one has $M_{n} = 0$ almost surely
for every $n$.

Our second result studies the evolution of two independent
$\smash {\widehat {S}}$-walks, denoted $\smash {\widehat {S}}^{1}$ and
$\smash {\widehat {S}}^{2}$.
\begin{theorem}
\label{teo:encounters}
Let $x_{1}, x_{2} \in \mathbb {Z}^{2} \setminus \{0\}$ have the same parity.
Then, we have
\begin{equation}
\label{eq:encounters}
\mathbb {P}_{x_{1},x_{2}} \big[\smash {\widehat {S}}_{n}^{1} =
\smash {\widehat {S}}_{n}^{2}\ \text{i.o.}\big] = 1.
\end{equation}
\end{theorem}

\begin{remark}
Theorem~\ref{teo:encounters} shows that almost surely two independent copies
of an $\smash {\widehat {S}}$-walk will meet infinitely many times. This
property holds for two independent SRWs since $(S^{1} - S^{2})_{n}$ is
a mean-zero random walk with bounded increments (see
\cite[Theorem~4.1.1]{lawler2010random}). Following the proof of Theorem~\ref{teo:encounters},
one could show analogous results when one of the walks is a SRW and the
other is an independent $\smash {\widehat {S}}$-walk or when both are SRWs.
We also notice that the problem of counting intersections of two independent
SRWs when we allow $S^{1}_{i} = S^{2}_{j}$ with $i \neq j$ is another well-studied
problem, cf. \cite{lawler2013intersections}.
\end{remark}

Finally, we mention that while developing the tools to prove Theorems~\ref{teo:speed}
and \ref{teo:encounters} we obtained a version of local Central Limit Theorem
for $\smash {\widehat {S}}_{n}$ which can be of independent interest.
\begin{proposition}
\label{prop:lclt}
If $1 \leqslant |x|, |y| \leqslant \sqrt{Mn}$ and $x-y$ has same parity
as $n$, then there is $n_{0}(M)$ positive such that for
$n \geqslant n_{0}$
\begin{equation}
\label{eq:lclt}
\mathbb {P}_{x}\bigl[ \smash {\widehat {S}}_{n} = y \bigr] \asymp
\frac{a(y)^{2}}{\ln ^{2} n} \times \frac{1}{n},
\end{equation}
where the implied constants in ``$\asymp $'' depend on $M > 1$ but not
on $n$, $x$, $y$. Moreover, there is a universal positive constant
$C$ such that for any $x, y \neq 0$ and any $n$ it holds
\begin{equation}
\label{eq:uniform_up_bound_lclt}
\mathbb {P}_{x} [\smash {\widehat {S}}_{n} = y] \leqslant \frac{C}{n}.
\end{equation}
\end{proposition}

Proposition~\ref{prop:lclt} says, in particular, that the expected number
of visits to a fixed site up to a given time converges, but very slowly.
Of course it has to converge since an $\smash {\widehat {S}}$-walk is transient
by construction and we recall that expression~\eqref{eq:conditional_green_function}
on Proposition~\ref{prop:properties_hS}.\textit{(vi)} is an exact expression
for the Green function of $\smash {\widehat {S}}$. Also, notice that the
uniform upper bound on~\eqref{eq:uniform_up_bound_lclt} is a property that
$\smash {\widehat {S}}$-walks share with SRWs.

The proof of each of the claims on Theorem~\ref{teo:speed} is obtained
by some variation of the Borel-Cantelli lemma. We use both the first and
second Borel-Cantelli lemmas and also a generalization known as the Kochen-Stone
theorem. These applications rely on being able to control the position
of a conditioned walk after $n$ steps. Although the proof of~\eqref{eq:teo_speed_eventually}
is quite straightforward from choosing a well-suited sequence of scales,
the proof of~\eqref{eq:teo_speed_io} needs more involved arguments to control
dependencies on the sequence of events.

The proof of Theorem~\ref{teo:encounters} uses the second moment method
on a sequence of random variables that count the number of encounters during
some well-separated time scales and a conditional Borel-Cantelli lemma.
In order to exemplify the proof on a more classical setting, we prove the
analogous result for SRW using this technique on Proposition~\ref{prop:SRW_recurrent}.

This paper is organized as follows. We begin Section~\ref{sec:auxiliary_results}
by presenting some properties already known for
$\smash {\widehat {S}}$-walks. In Sections~\ref{sec:controlling_the_position}
and~\ref{sec:local_clt_and_related_results} we develop new auxiliary tools
for analyzing the position of an $\smash {\widehat {S}}$-walk after
$n$ steps. In Section~\ref{sec:local_clt_and_related_results} we prove
Proposition~\ref{prop:lclt} and in Section~\ref{sec:controlling_the_position}
we provide tail estimates for the probability that the walk is too close
to the origin or too far away from it. Theorems~\ref{teo:speed} and~\ref{teo:encounters}
are proved in Sections~\ref{sec:speed_of_transience} and~\ref{sec:encounters_of_two_walks},
respectively.

\medskip
\noindent
\textbf{Notational remarks.} Throughout the paper we use $c, C$ to denote
generic positive constants that can change from line to line. Also, our
asymptotic notation uses
\begin{itemize}
\setlength{\itemsep }{1.5pt} \setlength{\parskip }{0pt}
\setlength{\parsep }{0pt}
\item both $f = o(g)$ and $f \ll g$ to denote
$\lim _{n \to \infty } \tfrac{f(n)}{g(n)} = 0$;
\item $f = O(g)$ to denote $|f| \leqslant C |g|$ for some constant
$C$;
\item both $f = \Theta (g)$ and $f \asymp g$ to denote
$c |g| \leqslant |f| \leqslant C |g|$;
\item $f \sim g$ to denote
$\lim _{n \to \infty } \tfrac{f(n)}{g(n)} = 1$.
\end{itemize}
If such a constant or asymptotic expression depends on other parameters
it will be made explicit. We use $\#A$ to denote the cardinality of set
$A$.

\section{Auxiliary results}
\label{sec:auxiliary_results}

In this section we collect some estimates for SRW and
$\smash {\widehat {S}}$-walks from other sources. In Sections~\ref{sec:controlling_the_position}
and~\ref{sec:local_clt_and_related_results} we develop other auxiliary
results that we will need.

A first useful result is to formalize how close is the law of an
$\smash {\widehat {S}}$-walk from that of a SRW conditioned on not hitting
the origin. To begin with, notice that if we know the endpoints
$x, y \neq 0$ of an $n$-step walk then if we sum over all possible paths
we get
\begin{equation}
\label{eq:paths_fixed_endpoints}
\mathbb {P}_{x} [\smash {\widehat {S}}_{n} = y] = \frac{a(y)}{a(x)}
\cdot \mathbb {P}_{x} [S_{n} = y, \tau (0) > n].
\end{equation}

A statement that works for more general events is Lemma~3.3 of
\cite{comets2016two}. Let $x \in B(L) \setminus \{0\}$ and denote by
$\Gamma ^{(x)}_{L}$ the set of all nearest-neighbor paths starting at
$x$ and ending at $\partial B(L)$. For $A \subset \Gamma ^{(x)}_{L}$, we
abuse notation and write $S \in A$ to denote the event that
$(S_{0}, \ldots , S_{k}) \in A$ for some $k$. We start with the following
estimate.
\begin{lemma}[Lemma 3.3\textit{(i)} of~\cite{comets2016two}]
\label{lema:cond_walk_similar_srw_not_hit_zero}
For any $A \subset \Gamma _{L}^{(x)}$ we have that
\begin{equation}
\mathbb {P}_{x} [S \in A \mid \tau (0) > \tau (L)] = \mathbb {P}_{x} [
\smash {\widehat {S}}\in A] (1 + O((L \ln L)^{-1})).
\end{equation}
\end{lemma}

Using Lemma~\ref{lema:cond_walk_similar_srw_not_hit_zero}, one can translate
estimates for SRW to estimates for $\smash {\widehat {S}}$-walk. The following
estimates for SRW are obtained using the Optional Stopping Theorem with
the martingale $a(S_{n \wedge \tau (0)})$.
\begin{lemma}[Lemma 3.1 of~\cite{comets2016two}]
\label{lema:srw_nothit_ball}
For all $x, y \in \mathbb {Z}^{2}$ and $L > 0$ with
$x \in B(y,L), |y| \leqslant L - 2$, we have
\begin{equation*}
\mathbb {P}_{x}\big[ \tau _{+}(0) > \tau _{+}\big(\partial B(y, L)
\big) \big] = \frac{a(x)}{a(L) + O\big(\frac{|y|\vee 1}{L}\big)}.
\end{equation*}
Moreover, for all $y \in B(n), x \in B(y, L)\setminus B(n)$ with
$n + |y| \leqslant L - 2$, we have
\begin{equation*}
\mathbb {P}_{x} \big[ \tau _{+}(n) > \tau _{+}\big(\partial B(y, L)
\big) \big] =
\frac{a(x) - a(n) +O(n^{-1})}{a(L) - a(n)
+ O\big(\frac{|y| \vee 1}{L} + n^{-1}\big)}.
\end{equation*}
\end{lemma}

Using Lemma~\ref{lema:cond_walk_similar_srw_not_hit_zero}, we can deduce
estimates for the $\smash {\widehat {S}}$-walk.
\begin{lemma}[Lemma 3.4 of~\cite{comets2016two}]
\label{lema:htau_r_before_htau_L}
For $r + 1 \leqslant |x| \leqslant L - 1$ we have that
\begin{equation*}
\mathbb {P}_{x}[\htau (r) > \htau (L)] = (1 + O(L^{-1}))
\frac{a(x) - a(r) + O(r^{-1})}{a(L) - a(r) + O(r^{-1})} \cdot
\frac{a(L)}{a(x)}.
\end{equation*}
Making $L \to \infty $, we get for $|x| \geqslant n + 1$ that
\begin{equation*}
\mathbb {P}_{x}[\htau (n) = \infty ] = 1 -
\frac{a(n) + O(n^{-1})}{a(x)}.
\end{equation*}
\end{lemma}

\section{On the position of conditioned walks}
\label{sec:controlling_the_position}

Now, we build on the results of the previous section to deduce more precise
bounds on probabilities of certain events related to the position of an
$\smash {\widehat {S}}$-walk. The general idea is that it is quite close
to a SRW when the walk is far away from the origin.

We provide two lemmas that will be constantly used throughout the paper.
The first lemma gives an upper estimate on how large is the hitting time
of a disk of large radius.
\begin{lemma}[Hitting times]
\label{lema:estimate_htau_r}
For simple random walk there are constants $c_{0}, r_{0} >0$ such that
for $r \geqslant r_{0}$ and $t \geqslant r^{2}$ we have
\begin{equation}
\label{eq:tail_estimate_tau_r}
\max _{x \in B(r)} \mathbb {P}_{x} [\tau (r) > t] \leqslant e^{- c_{0}
t r^{-2}}.
\end{equation}
Moreover, consider an $\smash {\widehat {S}}$-walk started from
$x \in B(r) \setminus \{0\}$. There is $c_{1} > 0$ such that, for
$r \geqslant r_{0}$ and
$t \geqslant \tfrac{2}{c_{0}} \cdot r^{2} (\ln \ln r)$, it holds that
\begin{equation}
\label{eq:tail_estimate_htau_r}
\mathbb {P}_{x} [\htau (r) > t] \leqslant \frac{1}{a(x)} \cdot \exp
\Bigl[ - c_{1} \cdot \frac{t}{r^{2}}\Bigr].
\end{equation}
\end{lemma}

\begin{proof}
Inequality~\eqref{eq:tail_estimate_tau_r} is shown using a coin-tossing
argument: for the event $\{\tau (r) > t\}$ to occur, it is necessary that
the walker fails to leave the disk $B(r)$ in each of several non-overlapping
time intervals. Indeed, let us divide the time interval $[0, t]$ into intervals
\begin{equation*}
\text{$I_{j} = [a_{j}, a_{j+1}]$ with
$a_{j} = (j - 1)\lfloor r^{2} \rfloor $,}
\end{equation*}
obtaining
$\bigl\lfloor \tfrac{t}{\lfloor r^{2} \rfloor } \bigr\rfloor $ intervals
(discard the last one if its length is less than
$\lfloor r^{2} \rfloor $). Notice that events
$A_{j} := \{\max _{u \in I_{j}} |S_{u}| < r \}$ satisfy
\begin{align*}
\max _{x} \mathbb {P}_{x}[\cap _{j=1}^{k} A_{j}] &= \max _{x}
\mathbb {E}_{x}[ \ind _{A_{1}} \mathbb {P}_{x}[ \{\cap _{j=1}^{k-1} A_{j}
\} \circ \theta _{a_{2}} \mid \mathcal {F}_{a_{2}} ]]
\\
&\leqslant \max _{x} \mathbb {P}_{x}[A_{1}] \max _{z} \mathbb {P}_{z}[
\cap _{j=1}^{k-1} A_{j}] \leqslant \ldots \leqslant \left (\max _{x}
\mathbb {P}_{x}[A_{1}]\right )^{k},
\end{align*}
where the maximums are on the disk $B(r)$ and we used Markov property at
time $a_{2}$. Moreover, we have
\begin{align*}
\max _{x} \mathbb {P}_{x}[A_{1}] &\leqslant \max _{x} \mathbb {P}_{x}
\bigl[ |S_{\lfloor r^{2} \rfloor }| < r\bigr] \leqslant \mathbb {P}_{0}
\bigl[ |S_{\lfloor r^{2} \rfloor }| < 3r\bigr]
\\
&= \mathbb {P}_{0} \Bigl[S_{\lfloor r^{2} \rfloor } \in B(3r)\Bigr]
\leqslant 1 - c_{0} < 1
\end{align*}
for large $r$ by the Central Limit Theorem. Thus, we can bound
\begin{equation*}
\max _{x \in B(r)} \mathbb {P}_{x} [\tau (r) > t] \leqslant \max _{x
\in B(r)} \mathbb {P}_{x} \bigl[\cap _{j=1}^{\lfloor t /\lfloor r^{2}
\rfloor \rfloor } A_{j}\bigr] \leqslant (1 - c_{0})^{\lfloor t /
\lfloor r^{2} \rfloor \rfloor } \leqslant e^{-c_{0} tr^{-2}}.
\end{equation*}

To get~\eqref{eq:tail_estimate_htau_r}, we just apply Lemmas~\ref{lema:cond_walk_similar_srw_not_hit_zero}
and \ref{lema:srw_nothit_ball}, obtaining
\begin{align*}
\mathbb {P}_{x} [\htau (r) > t] &= \mathbb {P}_{x}\big[\max _{u
\leqslant t} |\smash {\widehat {S}}_{u}| < r\big]
\\
&= \mathbb {P}_{x}\big[\max _{u \leqslant t} | S_{u}| < r \mid \tau (0)
> \tau (r)\big] (1 + O(r^{-1}))
\\
&\leqslant \mathbb {P}_{x}[\max _{u \leqslant t} |S_{u}| < r] \cdot
\mathbb {P}_{x} [\tau (0) > \tau (r)]^{-1} (1 + O(r^{-1}))
\\
&\leqslant e^{- c_{0} t r^{-2}} \cdot \frac{\ln r}{a(x)} \cdot (
\tfrac{2}{\pi } + o(1))
\\
&\leqslant a(x)^{-1} \cdot \exp \Bigl[ - c_{0} \cdot \frac{t}{r^{2}}
\Bigl( \frac{1}{2} + O\bigl(\ln ^{-1} \ln r\bigr) \Bigr) \Bigr]
\end{align*}
where in the last inequality we used that
$t \geqslant \tfrac{2}{c_{0}} \cdot r^{2} (\ln \ln r)$. Take
$c_{1} = \tfrac{c_{0}}{3}$ and increase $r_{0}$ if needed to ensure
$\tfrac{1}{2} + O(\ln ^{-1} \ln r) > \tfrac{1}{3}$.
\end{proof}

Since we expect that an $\smash {\widehat {S}}$-walk is close to a SRW,
it is reasonable to expect that after $n$ steps it is at distance roughly
$\sqrt{n}$. Using Lemma~\ref{lema:estimate_htau_r} we estimate the probability
of deviating too much from this.
\begin{lemma}
\label{lema:hS_in_extremes}
Let $x\in \mathbb {Z}^{2}\setminus \{0\}$ with $|x|\leqslant r$. Suppose
\begin{equation*}
n^{1/4} \leqslant r \leqslant n^{1/2} \text{ for all } n \quad
\text{and} \quad r \ll \sqrt{ \frac{n}{\ln \ln n} } \quad
\text{as $n \to \infty $}.
\end{equation*}
Choose any $f$ satisfying
$\tfrac{2}{c_{0}} \cdot \ln \ln r \leqslant f \ll \tfrac{n}{r^{2}}$ with
$c_{0}$ given by Lemma~\ref{lema:estimate_htau_r}. If
$u \geqslant 3 r$, there is $n_{0} > 0$ such that for
$n \geqslant n_{0}$ we have
\begin{equation*}
\label{eq:hS_greater_u}
\mathbb {P}_{x} \big[ |\smash {\widehat {S}}_{n}| > u \big]
\leqslant \frac{1}{a(x)} \cdot \exp \Bigl[ - c_{1} \cdot f \Bigr] + c_{2}
\exp \Bigl[ - \frac{u^{2}}{3n} \Bigr] + \frac{c_{3}}{\sqrt{n}}
\end{equation*}
for positive universal constants $c_{1},c_{2},c_{3}$. Also, for
$r \leqslant \tfrac{1}{3} l$ and $n \geqslant n_{0}$ we have a universal
constant $c_{4}>0$ such that
\begin{equation*}
\label{eq:hS_smaller_l}
\mathbb {P}_{x} \big[ |\smash {\widehat {S}}_{n}| < l \big]
\leqslant \frac{1}{a(x)} \cdot \exp \Bigl[ - c_{1} \cdot f \Bigr] + c_{4}
\cdot \frac{l^{2}}{n}.
\end{equation*}
\end{lemma}

\begin{proof}
All paths of conditioned walk till time $n$ must be inside
$B(n + r)$. We decompose our event with respect to stopping time
$\htau (r)$ the walk hits $\partial B(r)$, using Lemma~\ref{lema:estimate_htau_r}.
We can write
\begin{align*}
\mathbb {P}_{x} \bigl[ |\smash {\widehat {S}}_{n}| > u \bigr] &
\leqslant \mathbb {P}_{x} \bigl[ \htau (r) > t\bigr] + \mathbb {P}_{x}
\bigl[ |\smash {\widehat {S}}_{n}| > u, \htau (r) \leqslant t\bigr]
\\
&\leqslant \frac{1}{a(x)} \cdot e^{- c_{1} t r^{-2}} + \max _{
\substack{
z \in \partial B(r) \\
1 \leqslant j \leqslant t
}} \mathbb {P}_{z} \bigl[ |\smash {\widehat {S}}_{n-j}| > u \bigr].
\end{align*}
for any $t$ satisfying
$t \geqslant \tfrac{2}{c_{0}} \cdot r^{2} (\ln \ln r)$ and
$r \geqslant r_{0}$. Since $r \geqslant n^{1/4}$ we just need that
$n \geqslant n_{0}$, a universal constant. Taking $t = fr^{2}$, we obtain
\begin{equation*}
\mathbb {P}_{x} \bigl[ \htau (r) > t\bigr] \leqslant \exp \Bigl[ - c_{1}
\cdot f \Bigr] \cdot a(x)^{-1}.
\end{equation*}
Notice also that there is such an $f$ with $r^{2} f \ll n$. Indeed, the
hypothesis that $r \ll \sqrt{n} \cdot \ln ^{-1/2} \ln n$ implies
$r^{2} (\ln \ln r) \ll n$ as $n \to \infty $ and thus we can find an intermediary
$f$ satisfying the hypotheses. Now, for $z\in \partial B(r)$ we can compare
this event with SRW. Noticing that
$a(z) = \tfrac{2}{\pi } \ln r + O(1)$ and $r \in [n^{1/4}, n^{1/2}]$, we
have by Lemmas~\ref{lema:cond_walk_similar_srw_not_hit_zero} and~\ref{lema:srw_nothit_ball}
\begin{align*}
\mathbb {P}_{z} \bigl[ |\smash {\widehat {S}}_{n}| > u \bigr] &=
\mathbb {P}_{z} \bigl[ |S_{n}| > u \mid \tau (0) > \tau (n + r)
\bigr] \cdot (1 + O(n^{-1}))
\\
&\leqslant \mathbb {P}_{z} \bigl[ |S_{n}| > u \bigr] \cdot
\mathbb {P}_{z} \bigl[\tau (0) > \tau (n + r)\bigr]^{-1} \cdot (1 + O(n^{-1}))
\\
&= \mathbb {P}_{z} \bigl[ |S_{n}| > u \bigr] \cdot
\frac{a(n + r) + O\big( n^{-1} \big)}{a(z)} \cdot (1 + O(n^{-1}))
\\
&\leqslant (4 + o(1)) \cdot \mathbb {P}_{z} \bigl[ |S_{n}| > u \bigr].
\end{align*}
To estimate the corresponding quantity for SRW, we apply a Berry-Esseen
type of estimate for multidimensional walks from Bentkus
\cite{bentkus2005lyapunov}. For SRW, it states that uniformly on all convex
sets $A \subset \mathbb {R}^{2}$ we have
\begin{equation}
\label{eq:Bentkus_Berry_Esseen}
\mathbb {P}_{0} \big[ S_{n} \in A \big] = \mathbb {P}\big[
\mathcal {N}(0, \Sigma _{n}) \in A \big] + O(n^{- \frac{1}{2}})
\end{equation}
where
$\Sigma _{n} = \tfrac{n}{2} \big[ \substack{ 1 \\ 0 } \,
\substack{ 0 \\ 1 } \big]$ is the covariance matrix of $S_{n}$. Thus, using
$A = B(u)$ leads to
\begin{equation*}
\label{eq:Bentkus_Berry_Esseen_ball}
\mathbb {P}_{0} \big[ |S_{n}| > u \big] = 1 - \mathbb {P}\big[ |
\mathcal {N}(0, \Sigma _{n})| \leqslant u \big] + O(n^{-\frac{1}{2}}) =
e^{-\frac{u^{2}}{n}} + O(n^{-\frac{1}{2}}).
\end{equation*}
For a walk started from $z$, we notice that
\begin{align*}
\mathbb {P}_{z} \big[ |S_{n}| > u \big] &\leqslant \mathbb {P}_{0}
\big[ |S_{n}| > u - |z| \big] \leqslant \exp \biggl[ -\frac{u^{2}}{n}
\cdot \biggl( 1 - 2 \cdot \frac{r}{u} \biggr) \biggr] + O(n^{-
\frac{1}{2}}).
\end{align*}
Putting together all the bounds above, we get
\begin{align*}
\mathbb {P}_{x} \bigl[ |\smash {\widehat {S}}_{n}| > u \bigr] &
\leqslant \exp \Bigl[ - c_{1} \cdot f \Bigr] + (4 + o(1)) \exp
\biggl[ -\frac{u^{2}}{3n} \biggr] + O(n^{-\frac{1}{2}})
\end{align*}
and conclude the proof of the first result. The same idea is applied for
the second inequality since applying Lemma~\ref{lema:cond_walk_similar_srw_not_hit_zero}
to event $\{|\smash {\widehat {S}}_{n}| < l\}$ leads to
\begin{align*}
\mathbb {P}_{z} \bigl[ |\smash {\widehat {S}}_{n}| < l \bigr] &=
\mathbb {P}_{z} \bigl[ |S_{n}| < l \mid \tau (0) > \tau (n + r)
\bigr] \cdot (1 + o(1))
\\
&\leqslant (4 + o(1)) \cdot \mathbb {P}_{z} \bigl[ |S_{n}| < l \bigr].
\end{align*}
Applying \eqref{eq:Bentkus_Berry_Esseen} once again, notice that now it
holds for $r \leqslant \tfrac{1}{3} l$ that
\begin{align*}
\mathbb {P}_{z} \big[ |S_{n}| < l \big] &\leqslant \mathbb {P}_{0}
\big[ |S_{n}| < l + |z| \big] \leqslant 1 - \exp \Bigl[ -
\frac{l^{2}}{n} \Bigl( 1 + \frac{r}{l} \Bigr)^{2} \Bigr] + O(n^{-
\frac{1}{2}})
\\
&\leqslant 1 - \exp \Bigl[- \frac{16l^{2}}{9n}\Bigr] + O(n^{-
\frac{1}{2}}) = O \Bigl( \frac{l^{2}}{n} \Bigr)
\end{align*}
since
$\tfrac{l^{2}}{n} \geqslant \tfrac{9r^{2}}{n} \geqslant
\tfrac{9}{\sqrt{n}}$. Analogously to the previous case, we obtain
\begin{align*}
\mathbb {P}_{x} \big[ |\smash {\widehat {S}}_{n}| < l \big] &
\leqslant \frac{1}{a(x)} \cdot \exp \Bigl[ - c_{1} \cdot f \Bigr] + c_{4}
\cdot \frac{l^{2}}{n}. \qedhere
\end{align*}
\end{proof}

\begin{corollary}
\label{coro:hS_extremes_simple}
In the same setting of Lemma~\ref{lema:hS_in_extremes}, if
$r \ll n^{\beta }$ for some $\beta < \tfrac{1}{2}$, then as
$n \to \infty $ we have
\begin{equation*}
\label{eq:hS_extremes_simple}
\begin{split} \mathbb {P}_{x} \big[ |\smash {\widehat {S}}_{n}| > u
\big] &= O \Bigl( e^{ - c \cdot \frac{u^{2}}{n} } +
\tfrac{1}{\sqrt{n}} \Bigr)
\\
\mathbb {P}_{x} \big[ |\smash {\widehat {S}}_{n}| < l \big] &= O
\Bigl( \tfrac{l^{2}}{n} \Bigr).
\end{split}
\end{equation*}
\end{corollary}

\begin{proof}
Just notice that we can take $f = n^{1/2 - \beta }$ and for this choice
of $f$ we have
\begin{equation*}
\exp \bigl[ - c_{1} \cdot f \bigr] \ll n^{-1/2}. \qedhere
\end{equation*}
\end{proof}

\section{Local central limit theorem}
\label{sec:local_clt_and_related_results}

In this section we prove Proposition~\ref{prop:lclt}, our version of local
CLT for $\smash {\widehat {S}}$. It provides some finer estimates for the
position of an $\smash {\widehat {S}}$-walk after $n$ steps.

\begin{proof}[Proof of Proposition~\ref{prop:lclt}]
First of all, consider the time when the walk hits
$\partial B(n^{1/3})$. If $|x| < n^{1/3}$ we know by Lemma~\ref{lema:estimate_htau_r}
that for any $\delta > 0$ we have
\begin{equation*}
\label{eq:hS_confined_too_long}
\mathbb {P}_{x} [\htau (n^{1/3}) > n^{2/3 + \delta }] \leqslant e^{ - c
n^{\delta } } \ll \frac{1}{n \ln ^{2} n}.
\end{equation*}
Applying the Markov property to the stopping time
$\htau = \htau (n^{1/3})$ we can write
\begin{align}
\mathbb {P}_{x} \big[ \smash {\widehat {S}}_{n} = y \big] &=
\mathbb {P}_{x} \big[ \smash {\widehat {S}}_{n} = y, \htau < n^{2/3 +
\delta } \big] + \mathbb {P}_{x} \big[ \smash {\widehat {S}}_{n} = y,
\htau \geqslant n^{2/3 + \delta } \big]
\nonumber
\\
\label{eq:htau_decomposition}
&= \mathbb {E}_{x} \Big[
\hspace{-2mm}
\sum _{
\substack{
1\leqslant j \leqslant n^{2/3+\delta } \\
z \in \partial B(n^{1/3})
}}
\hspace{-3mm}
\ind \{ \htau = j, \smash {\widehat {S}}_{j} = z\} \mathbb {P}_{z}
\big[ \smash {\widehat {S}}_{n-j} = y \big] \Big] + O(e^{- c n^{
\delta }}).
\end{align}
This decomposition allows us to break the proof into cases, considering
how $|x|$ and $|y|$ compare to $n^{1/3}$.

\medskip
\noindent
\textbf{Uniform upper bound.} If $|x| < n^{1/3}$ we have from~\eqref{eq:htau_decomposition}
that
\begin{equation*}
\mathbb {P}_{x} [\smash {\widehat {S}}_{n} = y] \leqslant e^{- c n^{
\delta }} + \max _{
\substack{
z \in \partial B(n^{1/3}) \\
1 \leqslant j \leqslant n^{2/3 + \delta }
}} \mathbb {P}_{z} [\smash {\widehat {S}}_{n-j} = y].
\end{equation*}
Since $j \ll n$, it suffices to prove~\eqref{eq:uniform_up_bound_lclt}
for $|x| \geqslant n^{1/3}$. Recall relation~\eqref{eq:paths_fixed_endpoints}
and notice
\begin{equation*}
\mathbb {P}_{x} \bigl[\smash {\widehat {S}}_{n} = y\bigr] =
\frac{a(y)}{a(x)} \cdot \mathbb {P}_{x} [S_{n} = y, \tau (0) > n]
\leqslant \frac{a(y)}{a(x)} \cdot \mathbb {P}_{x} [S_{n} = y].
\end{equation*}
For SRW on the plane the uniform bound is straightforward from the local
CLT. To extend the result for an $\smash {\widehat {S}}$-walk, just notice
that $\frac{a(y)}{a(x)}$ is bounded by a constant. Indeed, using that
$|y-x| \leqslant n$ and the asymptotic expression for the potential kernel,
we have
\begin{equation*}
\frac{a(y)}{a(x)} = 1 +
\frac{\frac{2}{\pi } \ln \frac{|y|}{|x|} + O(1)}{a(x)} \leqslant 1 +
\frac{c \ln \left (1 + \frac{n}{n^{1/3}}\right ) + O(1)}{c \ln n}
\leqslant 1 + c \cdot \frac{\ln (2n^{2/3})}{\ln n}.
\end{equation*}

\medskip
Now, let us focus on proving~\eqref{eq:lclt}. Notice that~\eqref{eq:uniform_up_bound_lclt}
already proves the upper bound in~\eqref{eq:lclt} for the case
$|y|, |x| \geqslant n^{1/3}$.

\medskip
\noindent
\textbf{Lower bound case $|y|, |x| \geqslant n^{1/3}$.} Our proof of the
lower bound is more involved. We break it into two steps.
\begin{description}
\item[Step 1.] By allowing the walk to run for $\varepsilon n$ steps we
can consider points at distance of order $\sqrt{n}$ from the origin.
\item[Step 2.] We prove the lower bound for points at distance of order
$\sqrt{n}$ confined to a specific region of the plane. Then, we extend
to all points we need by concatenating regions.
\end{description}

\medskip
\noindent
\textbf{Step 1.} Let us consider the annulus
\begin{equation*}
A := \{w \in \mathbb {Z}^{2};\; 2 \varepsilon n^{1/2} \leqslant |w|
\leqslant Kn^{1/2}\}.
\end{equation*}
Then it is possible to choose $\varepsilon , K > 0$ depending on $M$ to
ensure that for every $z$ with
$n^{1/3} \leqslant |z| \leqslant \sqrt{Mn}$ we have
$ \mathbb {P}_{z} \big[ S_{\varepsilon n} \in A, \tau (0) >
\varepsilon n \big] \geqslant \delta $ for large $n$, where
$\delta $ is a positive constant.

\begin{proof}
Notice that by Lemma~\ref{lema:srw_nothit_ball} we can bound
\begin{equation*}
\mathbb {P}_{z} [\tau (0) > \varepsilon n] \geqslant \mathbb {P}_{z} [
\tau (0) > \tau (\partial B(z, \varepsilon n))] =
\frac{a(z)}{a(n) + O(|z|n^{-1})} \geqslant c > 0.
\end{equation*}
This means that the claim will follow once we prove
$\mathbb {P}_{z}[S_{\varepsilon n} \in A]$ can be made close to $1$. For
that, we once again use the Berry-Esseen result from Bentkus~\cite{bentkus2005lyapunov}.
Using equation~\eqref{eq:Bentkus_Berry_Esseen} twice we have
\begin{align*}
\mathbb {P}_{z}[S_{\varepsilon n}\!\in A] &= \mathbb {P}_{0} [S_{
\varepsilon n}\!\in B(K \sqrt{n}) - z] - \mathbb {P}_{0} [S_{
\varepsilon n}\!\in B(2 \varepsilon \sqrt{n}) - z]
\\
&= \mathbb {P}[ \mathcal {N}(z, \Sigma _{\varepsilon n})\!\in B(K
\sqrt{n})] - \mathbb {P}[\mathcal {N}(z, \Sigma _{\varepsilon n})\!
\in B(2 \varepsilon \sqrt{n})] + O\bigl(
\tfrac{1}{\sqrt{\varepsilon n}}\bigr)
\\
&= \mathbb {P}\bigl[ 2 \sqrt{\varepsilon } \leqslant \bigl|
\mathcal {N}\bigl(\tfrac{z}{\sqrt{\varepsilon n}}, \Sigma \bigr)
\bigr| \leqslant \tfrac{K}{\sqrt{\varepsilon }} \bigr] + O\bigl(
\tfrac{1}{\sqrt{\varepsilon n}}\bigr),
\end{align*}
where
$\Sigma _{N} = \tfrac{N}{2} \big[ \substack{ 1 \\ 0 } \,
\substack{ 0 \\ 1 } \big]$ is the covariance matrix of $S_{N}$ and
$\Sigma := \Sigma _{1}$. Let us denote by $f_{\bar{z}}(x)$ the density
of $\mathcal {N}(\bar{z}, \Sigma )$ with respect to the Lebesgue measure,
that is, $f_{\bar{z}}(x) = \tfrac{1}{\pi } e^{-|x - \bar{z}|^{2}}$. Since
$\bar{z} = \tfrac{z}{\sqrt{\varepsilon n}} \in B(\sqrt{M/\varepsilon })$,
we have that
\begin{equation*}
f_{\bar{z}}(x) = f_{0}(x) e^{- |\bar{z}|^{2} + 2 \langle x, \bar{z}
\rangle } \leqslant f_{0}(x) e^{4 \sqrt{M}} \quad
\text{for $|x| \leqslant 2 \sqrt{\varepsilon }$},
\end{equation*}
implying that, for every $\bar{z} \in B(\sqrt{M/\varepsilon })$,
\begin{equation*}
\mathbb {P}\bigl[|\mathcal {N}(\bar{z}, \Sigma )| \leqslant 2 \sqrt{
\varepsilon }\bigr] \leqslant \mathbb {P}\bigl[|\mathcal {N}(0,
\Sigma )| \leqslant 2 \sqrt{\varepsilon }\bigr] e^{4 \sqrt{M}} = (1 - e^{-4
\varepsilon }) e^{4 \sqrt{M}}.
\end{equation*}
On the other hand, we have
\begin{equation*}
\mathbb {P}\bigl[|\mathcal {N}(\bar{z}, \Sigma )| \geqslant
\tfrac{K}{\sqrt{\varepsilon }} \bigr] \leqslant \mathbb {P}\bigl[ |
\mathcal {N}(0, \Sigma )| \geqslant
\tfrac{K - \sqrt{M}}{\sqrt{\varepsilon }} \bigr] = e^{ -
\frac{(K - \sqrt{M})^{2}}{\varepsilon }}.
\end{equation*}
Fix $K > \sqrt{M}$. Making $n \to \infty $, we obtain
\begin{equation*}
\liminf _{n} \mathbb {P}_{z} [S_{\varepsilon n} \in A] \geqslant 1 - (1
- e^{-4 \varepsilon }) e^{4 \sqrt{M}} - e^{ -
\frac{(K - \sqrt{M})^{2}}{\varepsilon }}.
\end{equation*}
Thus, if we take $\varepsilon \downarrow 0$ the right hand side tends to
$1$.
\end{proof}

We can decompose our event with respect to the position of the walk on
times $\varepsilon n$ and $n - \varepsilon n$ and apply Step 1. We have
\begin{align}
\mathbb {P}_{x} &\big[ S_{n} = y, \tau (0) > n \big] \geqslant
\mathbb {P}_{x} \big[ S_{n} = y, \tau (0) > n, S_{\varepsilon n} \in A,
S_{n - \varepsilon n} \in A \big]
\nonumber
\\
&= \sum _{w, w' \in A} \mathbb {P}_{x} \big[ S_{\varepsilon n} = w,
\tau (0) > \varepsilon n \big] \times
\nonumber
\\[-2ex]
&%
\hspace{2cm}
\mathbb {P}_{w} \big[ S_{n - 2 \varepsilon n} = w', \tau (0) > n - 2
\varepsilon n \big] \times \mathbb {P}_{w'} \big[ S_{\varepsilon n} = y,
\tau (0) > \varepsilon n \big]
\nonumber
\\
\label{eq:low_bound_lclt_step1}
&\geqslant \delta ^{2} \inf _{w, w' \in A} \mathbb {P}_{w} \big[ S_{n -
2 \varepsilon n} = w', \tau (0) > n - 2 \varepsilon n \big]
\end{align}
and we just have to prove the lower bound for $w, w' \in A$.

\medskip
\noindent
\textbf{Step 2.} We claim that there is a positive constant $c > 0$ such
that for any $w, w' \in A$ we can write
\begin{equation*}
\mathbb {P}_{w} \big[ S_{n} = w', \tau (0) > n \big] \geqslant
\frac{c}{n}.
\end{equation*}

\begin{proof}
Let us decompose $A$ into 4 regions, by defining
\begin{equation*}
D_{1} := \{z \in A ;\; \langle z, (1, 1) \rangle \geqslant 2
\varepsilon \sqrt{n}\}
\end{equation*}
and $D_{i}$ for $i = 2, 3, 4$ are obtained from $D_{1}$ by rotating
$D_{1}$ around the origin with angles $\tfrac{\pi }{2}$, $\pi $ and
$\tfrac{3\pi }{2}$, respectively. Notice that regions $D_{i}$ cover the
set $A$ and $D_{i}$ intersects $D_{i+1}$ (see Figure~\ref{fig:Regions_low_bound}).
\begin{figure}
\centering
\def\svgwidth{.5\linewidth}
\begingroup%
  \makeatletter%
  \providecommand\color[2][]{%
    \errmessage{(Inkscape) Color is used for the text in Inkscape, but the package 'color.sty' is not loaded}%
    \renewcommand\color[2][]{}%
  }%
  \providecommand\transparent[1]{%
    \errmessage{(Inkscape) Transparency is used (non-zero) for the text in Inkscape, but the package 'transparent.sty' is not loaded}%
    \renewcommand\transparent[1]{}%
  }%
  \providecommand\rotatebox[2]{#2}%
  \newcommand*\fsize{\dimexpr\f@size pt\relax}%
  \newcommand*\lineheight[1]{\fontsize{\fsize}{#1\fsize}\selectfont}%
  \ifx\svgwidth\undefined%
    \setlength{\unitlength}{289.7849496bp}%
    \ifx\svgscale\undefined%
      \relax%
    \else%
      \setlength{\unitlength}{\unitlength * \real{\svgscale}}%
    \fi%
  \else%
    \setlength{\unitlength}{\svgwidth}%
  \fi%
  \global\let\svgwidth\undefined%
  \global\let\svgscale\undefined%
  \makeatother%
  \begin{picture}(1,0.69411491)%
    \lineheight{1}%
    \setlength\tabcolsep{0pt}%
    \put(0,0){\includegraphics[width=\unitlength,page=1]{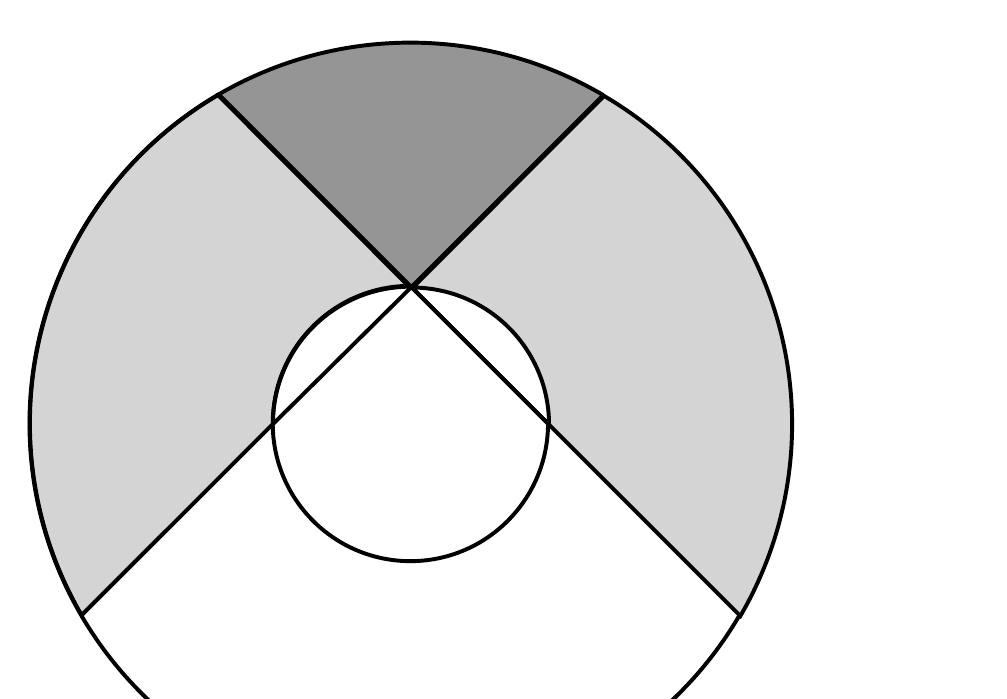}}%
    \put(0.58235363,0.29592541){\color[rgb]{0,0,0}\makebox(0,0)[lt]{\lineheight{1.25}\smash{\begin{tabular}[t]{l}$D_1 \setminus D_2$\end{tabular}}}}%
    \put(0.0800000,0.29592541){\color[rgb]{0,0,0}\makebox(0,0)[lt]{\lineheight{1.25}\smash{\begin{tabular}[t]{l}$D_2 \setminus D_1$\end{tabular}}}}%
    \put(0.32000000,0.54324269){\color[rgb]{0,0,0}\makebox(0,0)[lt]{\lineheight{1.25}\smash{\begin{tabular}[t]{l}$D_1 \cap D_2$\end{tabular}}}}%
    \put(0.35689751,0.08128168){\color[rgb]{0,0,0}\makebox(0,0)[lt]{\lineheight{1.25}\smash{\begin{tabular}[t]{l}$2\varepsilon \sqrt{n}$\end{tabular}}}}%
    \put(0.79170312,0.08128168){\color[rgb]{0,0,0}\makebox(0,0)[lt]{\lineheight{1.25}\smash{\begin{tabular}[t]{l}$K \sqrt{n}$\end{tabular}}}}%
    \put(0,0){\includegraphics[width=\unitlength,page=2]{Regions_low_bound.pdf}}%
  \end{picture}%
\endgroup%
\caption{Regions $D_i$ cover the set $A$. If $w, w' \in D_i$ we have a
    straightforward lower bound. For $w, w'$ in different regions $D_i$
    we concatenate paths.}
\label{fig:Regions_low_bound}
\end{figure}

We prove first that if $w, w' \in D_{1}$ then the lower bound holds. For
this, we use the fact that if we denote $S_{n} = (X_{n}, Y_{n})$ then the
processes $(X_{n} + Y_{n};\; n \geqslant 1)$ and
$(X_{n} - Y_{n} ;\; n \geqslant 1)$ are two independent SRWs on
$\mathbb {Z}$. Let $w, w' \in D_{1}$ and denote
\begin{align*}
a &= w_{1} + w_{2}, & b &= w_{1} - w_{2},
\\
a' &= w'_{1} + w'_{2}, & b' &= w'_{1} - w'_{2}.
\end{align*}
Using $S^{1}_{n}$ to denote the law of a one dimensional SRW we can write
\begin{align}
\mathbb {P}_{w} &\big[ S_{n} = w', \tau (0) > n \big] \geqslant
\mathbb {P}_{w} \big[ S_{n} = w', \langle S_{j}, (1,1) \rangle > 0 \
\text{for $j = 1, \ldots , n$} \big]
\nonumber
\\
&= \mathbb {P}_{a} \big[ S^{1}_{n} = a', \tau ^{1}(0) > n \big]
\cdot \mathbb {P}_{b} \big[ S^{1}_{n} = b' \big]
\nonumber
\\
&= \left ( \mathbb {P}_{a} \big[ S^{1}_{n} = a' \big] - \mathbb {P}_{a}
\big[ S^{1}_{n} = -a' \big] \right ) \cdot \mathbb {P}_{b} \big[ S^{1}_{n}
= b' \big]
\nonumber
\\
&= \mathbb {P}_{a} \big[ S^{1}_{n} = a' \big] \mathbb {P}_{b} \big[ S^{1}_{n}
= b' \big] \bigg( 1 -
\frac{\mathbb {P}_{a} \big[ S^{1}_{n} = -a' \big]}{\mathbb {P}_{a} \big[ S^{1}_{n} = a' \big]}
\bigg)
\nonumber
\\
\label{eq:2d_decoupling_on_D1}
&= \mathbb {P}_{w} \big[ S_{n} = w' \big] \bigg( 1 -
\frac{\mathbb {P}_{a} \big[ S^{1}_{n} = -a' \big]}{\mathbb {P}_{a} \big[ S^{1}_{n} = a' \big]}
\bigg),
\end{align}
where in the third line we used the reflection principle and in the last
line we used the independence of $(X_{n} + Y_{n})$ and
$(X_{n} - Y_{n})$ once again. Since we know $w, w' \in A$, the local CLT
in two dimensions provides
\begin{equation*}
\label{eq:1d_lclt_on_D1}
\mathbb {P}_{w} \big[ S_{n} = w' \big] = \frac{c}{n} \cdot e^{-
\frac{|w - w'|^{2}}{n}} + O(n^{-2}) \geqslant \frac{c}{n} \cdot e^{-4K^{2}}.
\end{equation*}
On the other hand, noticing that
$a, a' \geqslant 2 \varepsilon \sqrt{n}$ we have by one dimensional local
CLT that
\begin{align}
\frac{\mathbb {P}_{a} \big[ S^{1}_{n} = -a' \big]}{\mathbb {P}_{a} \big[ S^{1}_{n} = a' \big]}
&=
\frac{
\frac{c}{\sqrt{n}} e^{- \frac{(a + a')^{2}}{2n}} + O(n^{- \frac{3}{2}})
}{
\frac{c}{\sqrt{n}} e^{- \frac{(a - a')^{2}}{2n}} + O(n^{- \frac{3}{2}})
} = e^{ - \frac{1}{2n} \big( (a + a')^{2} - (a - a')^{2} \big) } + O(n^{-
1})
\nonumber
\\
\label{eq:bound_from_reflection}
&= e^{ - \frac{2a a'}{n} } + O(n^{- 1}) \leqslant e^{- 8 \varepsilon ^{2}}
+ O(n^{-1}).
\end{align}
Putting together \eqref{eq:2d_decoupling_on_D1}-\eqref{eq:bound_from_reflection},
we conclude the lower bound for $w, w' \in D_{1}$ and the same argument
works if $w, w'$ belong to the same $D_{i}$. In case
$w \in D_{1} \setminus D_{2}$ and $w' \in D_{2} \setminus D_{1}$, we notice
that for any $w'' \in D_{1} \cap D_{2}$ we have
\begin{equation*}
\mathbb {P}_{w} \big[ S_{n/2} = w'', \tau (0) > n/2 \big] \geqslant
\frac{c}{n} \quad \text{and} \quad \mathbb {P}_{w''} \big[ S_{n/2} = w',
\tau (0) > n/2 \big] \geqslant \frac{c}{n}.
\end{equation*}
Hence, we can write
\begin{align*}
\mathbb {P}_{w} &\big[ S_{n} = w', \tau (0) > n \big]
\\
&\geqslant \sum _{w'' \in D_{1} \cap D_{2}} \mathbb {P}_{w} \big[ S_{n/2}
= w'', \tau (0) > n/2 \big] \mathbb {P}_{w''} \big[ S_{n/2} = w',
\tau (0) > n/2 \big]
\\
&\geqslant \# (D_{1} \cap D_{2}) \cdot c n^{-2} \geqslant \frac{c}{n}.
\end{align*}
Applying the above reasoning at most twice we prove that the lower bound
holds for every $w, w' \in A$.
\end{proof}

Combining Step~2 with \eqref{eq:low_bound_lclt_step1} we conclude the proof
of the lower bound in the case
$n^{1/3} \leqslant |x|, |y| \leqslant \sqrt{Mn}$. Since we already have
the upper bound for this case, all that remains is to extend~\eqref{eq:lclt}
to the other cases using the decomposition in~\eqref{eq:htau_decomposition}.

\medskip
\noindent
\textbf{Case $|x| \leqslant n^{1/3}$ and $n^{1/3} \leqslant |y|$.} Recalling~\eqref{eq:htau_decomposition}
we have
\begin{equation*}
\mathbb {P}_{x} \big[ \smash {\widehat {S}}_{n} = y \big] =
\mathbb {E}_{x} \Big[
\hspace{-2mm}
\sum _{
\substack{
1\leqslant j \leqslant n^{2/3+\delta } \\
z \in \partial B(n^{1/3})
}}
\hspace{-3mm}
\ind \{ \htau = j, \smash {\widehat {S}}_{j} = z\} \mathbb {P}_{z}
\big[ \smash {\widehat {S}}_{n-j} = y \big] \Big] + O(e^{- c n^{
\delta }}).
\end{equation*}
Since $j \leqslant n^{2/3 + \delta } \ll n$ and
$n^{1/3} \leqslant |y|, |z| \leqslant \sqrt{Mn}$, we conclude from the
previous case that
\begin{equation*}
\mathbb {P}_{x} \big[ \smash {\widehat {S}}_{n} = y \big] = \big( 1 - O
\big(e^{- c n^{\delta }}\big) \big) \cdot \Theta (n^{-1}) + O \big( e^{-
c n^{\delta }} \big) = \Theta (n^{-1}).
\end{equation*}

\medskip
\noindent
\textbf{Case $|y| \leqslant n^{1/3}$ and $n^{1/3} \leqslant |x|$.} This
case is straightforward from the reversibility of the
$\smash {\widehat {S}}$-walk, since
\begin{equation*}
\mathbb {P}_{x} \big[ \smash {\widehat {S}}_{n} = y \big] =
\frac{a(y)^{2}}{a(x)^{2}} \mathbb {P}_{y} \big[ \smash {\widehat {S}}_{n}
= x \big] \asymp \frac{a(y)^{2}}{\ln ^{2} n} \cdot \frac{1}{n}.
\end{equation*}

\medskip
\noindent
\textbf{Case $1 \leqslant |x|, |y| \leqslant n^{1/3}$.} Apply~\eqref{eq:htau_decomposition}
and use the previous case to obtain
$\mathbb {P}_{z} \bigl[ \smash {\widehat {S}}_{n-j} = y \bigr]$.
\end{proof}

\section{Speed of transience}
\label{sec:speed_of_transience}

In this section we prove Theorem~\ref{teo:speed}. Let us recall that
$M_{n} := \min _{m \geqslant n} |\smash {\widehat {S}}_{m}|$.

\medskip
\noindent
\textbf{Eventually
$M_{n} \leqslant \sqrt{(e + \delta )n (\ln \ln n)}$.} Pick
$0 < \delta ' < \tfrac{\delta }{e}$. By Corollary~\ref{coro:hS_extremes_simple}
\begin{equation*}
\mathbb {P}_{x} \big[|\smash {\widehat {S}}_{n}| > \sqrt{(1+\delta ')n
(\ln \ln n)}\big] \leqslant c e^{- (1+\delta ')(\ln \ln n)} =
\frac{c}{\ln ^{1+\delta '} n},
\end{equation*}
and for the sequence $n_{k} = e^{k}$ we have
\begin{equation*}
\sum _{k\geqslant 1} \frac{1}{\ln ^{1+\delta '} n_{k}} = \sum _{k
\geqslant 1} \frac{1}{k^{1+\delta '}} < \infty .
\end{equation*}
The first Borel-Cantelli lemma implies we have
$|\smash {\widehat {S}}_{n_{k}}| \leqslant \sqrt{(1+\delta ')n_{k} (
\ln \ln n_{k})}$ eventually and so does
$M_{n_{k}} \leqslant |\smash {\widehat {S}}_{n_{k}}|$. We can conclude
the result for every $n$ by noticing that if
$k(t) = \lceil \ln t \rceil $ then
\begin{align*}
M_{t} &\leqslant M_{n_{k(t)}} \leqslant \sqrt{(1+\delta ')e^{k(t)} (
\ln k(t))}
\\
&\leqslant \sqrt{(1+\delta ')e^{\ln t + 1} (\ln (\ln t + 1))}
\\
&= \sqrt{(1+\delta ')te \big(\ln \ln t + \ln \big(1 +
\tfrac{1}{\ln t}\big)\big)}
\\
&= \sqrt{(e+e\delta ')t (\ln \ln t) \cdot \big( 1 + o(1) \big)}
\\
&\leqslant \sqrt{(e+\delta )t (\ln \ln t)}
\end{align*}
for large $t$.

\medskip
\noindent
\textbf{Infinitely often $M_{n} \leqslant n^{\delta }$.} Define
\begin{equation*}
\htau _{n} = \inf \{ t > 0; \smash {\widehat {S}}_{t} \in \partial B(n^{
\delta }) \cup \partial B(n) \}
\end{equation*}
and let $\theta _{t}$ be a time shift of $t$ for the Markov chain
$\smash {\widehat {S}}_{n}$. The stopping times
\begin{align*}
U_{0} &:= \inf \{ n\geqslant 1; |\smash {\widehat {S}}_{n}|
\leqslant n^{\frac{1}{2} + \delta }\}, \quad & & U_{i} := \inf \{ n> V_{i-1};
\; |\smash {\widehat {S}}_{n}| \leqslant n^{\frac{1}{2} + \delta }\}
\\
V_{0} &:= \htau _{U_{0}} \circ \theta _{U_{0}}, \quad & & V_{i} :=
\htau _{U_{i}} \circ \theta _{U_{i}}.
\end{align*}
are well defined since the previous result implies that
$|\smash {\widehat {S}}_{n}| \leqslant n^{\frac{1}{2} + \delta }$ infinitely
often. Notice that all of them are finite, we have
$U_{i} < V_{i} < U_{i+1}$, and
$|\smash {\widehat {S}}_{U_{i}}| \leqslant U_{i}^{\frac{1}{2} +
\delta }$. Moreover, if $\mathcal {F}_{U_{i}}$ is the smallest
$\sigma $-algebra that makes $\smash {\widehat {S}}_{U_{i}}$ measurable,
we have
\begin{equation*}
\mathbb {P}[\smash {\widehat {S}}_{V_{i}} \in \partial B(U_{i}^{
\delta }) \mid \mathcal {F}_{U_{i}}] \geqslant c,
\end{equation*}
a positive constant. Indeed, if
$|\smash {\widehat {S}}_{U_{i}}| \leqslant U_{i}^{\delta }$ this probability
is one. On the other hand, if
$|\smash {\widehat {S}}_{U_{i}}| > U_{i}^{\delta }$ the lower bound follows
from Lemma~\ref{lema:htau_r_before_htau_L} since for
$|x| \in (n^{\delta }, n^{1/2+\delta }]$ we have
\begin{align*}
\mathbb {P}_{x} [\htau (n^{\delta }) < \htau (n)] &= 1 -
\frac{%
a(x) - a(n^{\delta }) + O(n^{-\delta })
}{%
a(n) - a(n^{\delta }) + O(n^{-\delta })
} \cdot \frac{a(n)}{a(x)} \cdot (1 + O(n^{-1}))
\\
&= 1 - \Bigl(1 - \frac{a(n^{\delta })}{a(x)} + o(1)\Bigr) \cdot
\frac{1}{1 - \delta } \cdot (1 + o(1))
\\
&\geqslant 1 - \Bigl(1 -
\frac{\delta \ln n + O(1)}{(\frac{1}{2}+\delta ) \ln n + O(1)} + o(1)
\Bigr) \cdot \Bigl(\frac{1}{1 - \delta } + o(1)\Bigr)
\\
&= 1 - \Bigl(1 - \frac{2\delta }{1+2\delta }\Bigr) \cdot
\frac{1}{1 - \delta } + o(1)
\\
&= \frac{\delta (1 - 2\delta )}{(1+2\delta )(1-\delta )} + o(1)
\end{align*}
which is positive for large $n$. We used
$a(x) \leqslant \frac{2}{\pi } (\frac{1}{2}+\delta )\ln n + O(1)$ in the
only inequality. Since these events are independent, we have by the second
Borel-Cantelli lemma that
$|\smash {\widehat {S}}_{V_{i}}| \leqslant U_{i}^{\delta }$ infinitely
often, and thus $M_{n} \leqslant n^{\delta }$ infinitely often.

\medskip
\noindent
\textbf{Eventually $M_{n} \geqslant e^{\ln ^{1 - \delta } n}$.} Let us define
\begin{equation*}
\sigma _{y} := \min \{n \geqslant 0;\; |\smash {\widehat {S}}_{n}|
\geqslant y \} \quad \text{and} \quad \eta _{x} := \max \{n \geqslant 0;
\; |\smash {\widehat {S}}_{n}| \leqslant x \}.
\end{equation*}
For any $y > x > 0$ we can write
\begin{equation*}
\mathbb {P}[\eta _{x} > n] = \mathbb {P}[\eta _{x} > n, \sigma _{y}
\leqslant n] + \mathbb {P}[\eta _{x} > n, \sigma _{y} > n].
\end{equation*}
The second term in this sum is small if we consider
$y = n^{1/2 - \delta }$, since
\begin{equation*}
\mathbb {P}[\sigma _{n^{1/2 - \delta }} > n] = \mathbb {P}\bigl[\max _{t
\leqslant n} |\smash {\widehat {S}}_{t}| < n^{1/2 - \delta }\bigr] = O
\big( e^{- c n^{\delta }} \big)
\end{equation*}
by Lemma~\ref{lema:estimate_htau_r}. On the other hand, we can estimate
the first term by
\begin{equation*}
\mathbb {P}\bigl[\eta _{x} > n, \sigma _{y} \leqslant n\bigr]
\leqslant \mathbb {P}\bigl[\eta _{x} > \sigma _{y}, \sigma _{y}
\leqslant n\bigr] \leqslant \mathbb {P}_{y} \bigl[ \htau (x) <
\infty \bigr] \sim \frac{\ln x}{\ln y},
\end{equation*}
implying that
\begin{equation*}
\mathbb {P}\bigl[\eta _{x} > n\bigr] \leqslant c_{\delta } \cdot
\frac{\ln x}{\ln n}
\end{equation*}
as long as $e^{-c n^{\delta }} \ll \frac{\ln x}{\ln n}$. Consider sequences
$x_{k}=e^{k^{m}}$ and $n_{k}=e^{k^{m+2}}$ where~$m$ is a parameter we choose
later. Then, we have
\begin{equation*}
\sum _{k\geqslant 1} \mathbb {P}\bigl[\eta _{x_{k}} > n_{k}\bigr]
\leqslant \sum _{k\geqslant 1} \frac{k^{m}}{k^{m+2}} < \infty \quad
\text{implies} \quad \mathbb {P}\bigl[\eta _{x_{k}} \leqslant n_{k},
\text{ eventually}\bigr] = 1.
\end{equation*}
This means that a.s.\ there is a random $k_{0}$ such that for every
$k \geqslant k_{0}$ and $n \geqslant n_{k}$ it holds
$|\smash {\widehat {S}}_{n}| \geqslant x_{k}$. Notice that we can write
\begin{equation*}
x_{k} = \exp [k^{m}] = \exp [ \ln ^{\frac{m}{m+2}} n_{k}] = \exp [
\ln ^{1 - \frac{2}{m+2}} n_{k}].
\end{equation*}
Thus, we have
$M_{n_{k}} \geqslant \exp [ \ln ^{1 - \frac{2}{m+2}} n_{k}]$ for
$k \geqslant k_{0}$. Since $M_{n}$ is non-decreasing we can extend the
result for every large $n$. Indeed, for any $t\in \mathbb {N}$ let
$k(t)$ be the largest integer such that $n_{k(t)} \leqslant t$, i.e.,
$k(t) := \lfloor \ln ^{\frac{1}{m+2}} t \rfloor $. We have
\begin{align*}
M_{t} &\geqslant M_{n_{k(t)}} \geqslant x_{k(t)} = \exp \Big[ \Big(
\lfloor \ln ^{\frac{1}{m+2}} t \rfloor \Big)^{m} \Big]
\\
&\geqslant \exp \Big[ \Big( \ln ^{\frac{1}{m+2}} t - 1 \Big)^{m}
\Big]
\\
&\geqslant \exp \Big[ \ln ^{1-\frac{2}{m+2}} t - m\ln ^{1-
\frac{3}{m+2}} t \Big].
\end{align*}
Since $m$ is arbitrary, we can choose $m$ so that
$\frac{2}{m+2} < \delta $. Denoting $s = \ln t$, we have that
$M_{t} \geqslant \exp [ \ln ^{1-\delta } t ]$ if
\begin{equation*}
s^{1-\frac{2}{m+2}} - ms^{1-\frac{3}{m+2}} \geqslant s^{1-\delta }
\quad \text{or equivalently} \quad s^{\delta -\frac{2}{m+2}} - ms^{
\delta -\frac{3}{m+2}} \geqslant 1,
\end{equation*}
which holds for large $t$.

\medskip
\noindent
\textbf{Infinitely often
$M_{n} \geqslant \frac{\sqrt{n}}{\ln ^{\delta } n}$.} In general, at time
$n$ we have that the walk should be at a distance of order~$\sqrt{n}$ from
the origin. We consider three breaking points near $\sqrt{n}$:
\begin{equation}
\label{eq:defi_u_l_m}
u(n) = \sqrt{n \ln ^{2} \ln n} \quad \text{and} \quad l(n) =
\frac{\sqrt{n}}{\ln \ln n} \quad \text{and} \quad m(n) =
\frac{\sqrt{n}}{\ln ^{\delta } n}.
\end{equation}
The letters above are chosen as mnemonics for upper, lower and minimum,
respectively, and we omit their dependence on $n$. By Corollary~\ref{coro:hS_extremes_simple}
we know
\begin{equation}
\label{eq:KochenStone_extremes}
\mathbb {P}_{x} \bigl[|\smash {\widehat {S}}_{n}| < l \bigr]
\leqslant \frac{c}{\ln ^{2} \ln n} \quad \text{and} \quad \mathbb {P}_{x}
\bigl[|\smash {\widehat {S}}_{n}| > u \bigr] \leqslant
\frac{c}{e^{\ln ^{2} \ln n}}.
\end{equation}
We notice that on the highly probable event
$\{|\smash {\widehat {S}}_{n}| \in [l, u]\}$ the probability of avoiding
$B(m)$ is almost independent of the precise position. Indeed, for
$|z| \in [l, u]$ our Lemma~\ref{lema:htau_r_before_htau_L} states that
\begin{align}
\mathbb {P}_{z} [\htau (m) = \infty ] &=
\frac{a(z) - a(m) + O(m^{-1})}{a(z)}
\nonumber
\\
&=
\frac{%
\frac{1}{2} \ln n + O(\ln \ln \ln n) -
\frac{1}{2} \ln n + \delta \ln \ln n
}{%
\frac{1}{2} \ln n + O(\ln \ln \ln n)
}
\nonumber
\\
\label{eq:KochenStone_const_hitting}
&= (2\delta + o(1)) \frac{\ln \ln n}{\ln n}.
\end{align}
The estimate~\eqref{eq:KochenStone_const_hitting} is essentially the probability
of $\mathbb {P}_{x}[M_{n} \geqslant m]$. Indeed, by~\eqref{eq:KochenStone_const_hitting}
we can see that a lower bound for
$\mathbb {P}_{x} [M_{n} \geqslant m]$ is
\begin{align*}
\mathbb {P}_{x} [M_{n} \geqslant m] &\geqslant \mathbb {P}_{x} \bigl[ |
\smash {\widehat {S}}_{n}| \geqslant l \bigr] \cdot \mathbb {P}_{x}
\bigl[ M_{n} \geqslant m \ \bigm|\ |\smash {\widehat {S}}_{n}|
\geqslant l \bigr]
\\
&\geqslant (1+o(1)) \cdot \mathbb {P}_{l} \bigl[\htau (m) = \infty
\bigr]
\\
&= 2\delta \frac{\ln \ln n}{\ln n} (1+o(1)),
\end{align*}
and a similar computation leads to the same upper bound
\begin{align*}
\mathbb {P}_{x} [M_{n} \geqslant m] &\leqslant \mathbb {P}_{x} \bigl[ |
\smash {\widehat {S}}_{n}| \in [m, u] \bigr] \!\cdot \mathbb {P}_{x}
\bigl[ M_{n} \geqslant m \!\bigm| |\smash {\widehat {S}}_{n}| \in [m, u]
\bigr] + \mathbb {P}_{x} \bigl[ |\smash {\widehat {S}}_{n}| > u
\bigr]
\\
&\leqslant 1 \cdot \mathbb {P}_{u} \bigl[\htau (m) = \infty \bigr] + c
\cdot e^{-\ln ^{2} \ln n}
\\
&= (2\delta + o(1)) \frac{\ln \ln n}{\ln n}.
\end{align*}

The idea now is to apply a generalization of Borel-Cantelli Lemma, known
as the Kochen-Stone theorem~\cite{kochen1964note}. It asserts that for
any sequence of events~$A_{n}$ such that
$\sum \mathbb {P}[A_{n}] = \infty $ it holds
\begin{equation}
\label{eq:KochenStone_theorem}
\mathbb {P}[A_{n}\ \text{i.o.}] \geqslant \limsup _{k}
\frac{
\sum _{i_{0} \leqslant i < j \leqslant k} \mathbb {P}[A_{i}] \mathbb {P}[A_{j}]
}{
\sum _{i_{0} \leqslant i < j \leqslant k} \mathbb {P}[A_{i} \cap A_{j}]
},
\end{equation}
with $i_{0}$ being any finite starting point. Let
\begin{equation*}
A_{k} := \bigl\{M_{n_{k}} \geqslant
\tfrac{\sqrt{n_{k}}}{\ln ^{\delta } n_{k}} \bigr\} \quad
\text{with $n_{k} = e^{k \ln ^{2} k}$}
\end{equation*}
and define quantities $u_{k}, l_{k}, m_{k}$ like in~\eqref{eq:defi_u_l_m}
for $n = n_{k}$. This choice of $n_{k}$ ensures that $n_{k}$ grows quickly,
which makes events $A_{k}$ closer to independent, while keeping
$\sum \mathbb {P}_{x}[A_{j}]$ divergent since
\begin{align*}
\sum _{j=1}^{k} \mathbb {P}_{x} [A_{j}] &= (2\delta + o(1)) \sum _{j=1}^{k}
\frac{\ln j + 2 \ln \ln j}{j \ln ^{2} j} = (2\delta + o(1)) \sum _{j=1}^{k}
\frac{1}{j \ln j}
\\
&= (2\delta + o(1)) \ln \ln k.
\end{align*}
Also, it is simple to check that
$m_{k} \ll l_{k} \ll u_{k} \ll m_{k+1}$ as $k \to \infty $. For the numerator
on~\eqref{eq:KochenStone_theorem} we get
\begin{align}
\sum _{i_{0} \leqslant i < j \leqslant k} \mathbb {P}_{x} [A_{i}]
\mathbb {P}_{x} [A_{j}] &= \frac{1}{2} \biggl[ \biggl(\sum _{j=i_{0}}^{k}
\mathbb {P}_{x} [A_{j}]\biggr)^{2} - \sum _{j=i_{0}}^{k} \mathbb {P}_{x}
[A_{j}]^{2} \biggr]
\nonumber
\\
\label{eq:numerator_Kochen_Stone}
&= (2 \delta ^{2} + o(1)) \ln ^{2} \ln k.
\end{align}

\medskip
\noindent
\textbf{Estimating the denominator.} We need to show that the denominator
of~\eqref{eq:KochenStone_theorem} has the same asymptotic behavior as in~\eqref{eq:numerator_Kochen_Stone}.
We firstly prove
\begin{lemma}
\label{lema:KochenStone_denom_simpler}
It holds that
\begin{equation*}
\label{eq:KochenStone_denom_simpler}
\!\sum _{i_{0} \leqslant i < j \leqslant k}\!\!\! \mathbb {P}_{x}[A_{i}
\cap A_{j}] =\!\!\!\!\sum _{i_{0} \leqslant i < j \leqslant k}\!\!\!
\mathbb {P}_{u_{i}}[\htau (l_{j}) < \htau (m_{i})] \cdot \mathbb {P}_{u_{j}}[
\htau (m_{j}) = \infty ] + o\bigl(\ln ^{2} \ln k\bigr).
\end{equation*}
\end{lemma}
Recall that by~\eqref{eq:KochenStone_const_hitting} we have
\begin{equation*}
\mathbb {P}_{x}[A_{j}] = (1+o(1)) \cdot \mathbb {P}_{u_{j}}[\htau (m_{j})
= \infty ] = \frac{2\delta + o(1)}{j \ln j}.
\end{equation*}
Using Lemma~\ref{lema:htau_r_before_htau_L}, we have a similar estimate
for the term
\begin{align}
\mathbb {P}_{u_{i}} &[\htau (l_{j}) < \htau (m_{i})]
\nonumber
\\
&= (1 + O(l_{j}^{-1})) \cdot
\frac{a(u_{i}) - a(m_{i}) + O(m_{i}^{-1})}{a(l_{j}) - a(m_{i}) + O(m_{i}^{-1})}
\cdot \frac{a(l_{j})}{a(u_{i})}
\nonumber
\\
&= (1 + o(1)) \cdot
\frac{a(u_{i}) - a(m_{i}) + O(m_{i}^{-1})}{a(u_{i})} \cdot
\frac{a(l_{j})}{a(l_{j}) - a(m_{i}) + O(m_{i}^{-1})}
\nonumber
\\
\label{eq:KochenStone_perturbation}
&= \frac{(2\delta + o(1))}{i \ln i} \cdot
\frac{j \ln ^{2} j}{j \ln ^{2} j - i \ln ^{2} i}.
\end{align}
After proving Lemma~\ref{lema:KochenStone_denom_simpler}, the proof that
$M_{n} \geqslant \tfrac{\sqrt{n}}{\ln ^{\delta } n}$ infinitely often will
be complete using the following estimate.
\begin{lemma}
\label{lema:double_sum_estimate}
We have that
\begin{equation*}
\label{eq:double_sum_estimate}
\sum _{i_{0} \leqslant i < j \leqslant k}\!\! \frac{1}{i \ln i}
\cdot \frac{\ln j}{j \ln ^{2} j - i \ln ^{2} i} \leqslant \left (
\frac{1}{2} + o(1)\right ) \cdot \ln ^{2} \ln k.
\end{equation*}
\end{lemma}

\begin{proof}[Proof of Lemma~\ref{lema:KochenStone_denom_simpler}]
For $i < j$ we decompose event $A_{i} \cap A_{j}$ with respect to the position
of $\smash {\widehat {S}}_{n_{i}}$, considering the intervals
\begin{equation*}
I_{1} = [m_{i}, u_{i}] \quad \text{and} \quad I_{2} = \bigl(u_{i}, n_{j}^{1/3}
\bigr] \quad \text{and} \quad I_{3} = \bigl(n_{j}^{1/3}, \infty \bigr).
\end{equation*}
We begin by showing that by restricting to intervals $I_{2}$ and
$I_{3}$ we obtain upper bounds of smaller order than what we need.

\medskip
\noindent
\textbf{Interval $I_{3}$.} Notice that after $n_{i}$ steps an
$\smash {\widehat {S}}$-walk can be at most at distance
$|x| + n_{i}$. Also, if $j = 3i$ we have that
\begin{align*}
n_{j}^{1/3} &= \exp \bigl[ \tfrac{1}{3} (3i) \ln ^{2} (3i) \bigr] =
\exp \bigl[ i \ln ^{2} i + (2 (\ln 3) + o(1)) \cdot i \ln i \bigr]
\gg n_{i},
\end{align*}
implying that, for every $i \geqslant i_{0}(x)$,
$\smash {\widehat {S}}_{n_{i}}$ cannot reach
$\partial B(n_{j}^{1/3})$ if $j \geqslant 3i$. Thus, we can write
\begin{align*}
\!\!\sum _{i_{0} \leqslant i < j \leqslant k}\!\!\! \!\mathbb {P}_{x} [A_{i}
\cap A_{j}, |\smash {\widehat {S}}_{n_{i}}| \in I_{3}] &= \sum _{i = i_{0}}^{k-1}
\sum _{j=i+1}^{(3i) \wedge k} \!\mathbb {P}_{x} [A_{i} \cap A_{j}, |
\smash {\widehat {S}}_{n_{i}}| \in I_{3}] \leqslant \sum _{i = i_{0}}^{k-1}
\frac{c \cdot (2i)}{e^{ \ln ^{2} \ln n_{i} }}
\end{align*}
which is finite since we know that
$\ln \ln n_{i} = (\ln i) (1 + o(1))$.

\medskip
\noindent
\textbf{Interval $I_{2}$.} We have by Markov property at times
$n_{i}$ and $n_{j} - n_{i}$ and the estimates in~\eqref{eq:KochenStone_extremes}
that
\begin{align*}
&\mathbb {P}_{x}[A_{i} \cap A_{j} \mid |\smash {\widehat {S}}_{n_{i}}|
\in I_{2}]
\\
&\leqslant \max _{|z| \in I_{2}} \left \{ \mathbb {P}_{z} \bigl[ |
\smash {\widehat {S}}_{n_{j} - n_{i}}| \in [m_{j}, u_{j}], \{\htau (m_{j})
= \infty \} \circ \theta _{n_{j} - n_{i}}\bigr] + \mathbb {P}_{z}
\bigl[|\smash {\widehat {S}}_{n_{j} - n_{i}}| > u_{j} \bigr] \right
\}
\\
&\leqslant \max _{|z| \in I_{2}} \mathbb {P}_{z} [|
\smash {\widehat {S}}_{n_{j} - n_{i}}| \in [m_{j}, u_{j}]] \cdot \!\!
\max _{|w| \in [m_{j}, u_{j}]} \mathbb {P}_{w} [ \htau (m_{j}) =
\infty ] + \frac{c}{e^{\ln ^{2} \ln n_{j} (1 + o(1))}}
\\
&\leqslant 1 \cdot \mathbb {P}_{u_{j}}[\htau (m_{j}) = \infty ] +
\frac{c}{e^{\ln ^{2} \ln n_{j} (1 + o(1))}}.
\end{align*}
Thus, using~\eqref{eq:KochenStone_const_hitting} we can write
\begin{align*}
\mathbb {P}_{x}\bigl[A_{n_{i}} \cap A_{n_{j}}, |\smash {\widehat {S}}_{n_{i}}|
\in I_{2}\bigr] &= \mathbb {P}_{x}\bigl[|\smash {\widehat {S}}_{n_{i}}|
\in I_{2}\bigr] \cdot \mathbb {P}_{x}\bigl[A_{n_{i}} \cap A_{n_{j}}
\mid |\smash {\widehat {S}}_{n_{i}}| \in I_{2}\bigr]
\\
&\leqslant \frac{c}{e^{\ln ^{2} \ln n_{i}}} \cdot (2\delta + o(1))
\frac{\ln \ln n_{j}}{\ln n_{j}}.
\end{align*}
Summing for $i < j$ and using that $e^{-\ln ^{2} \ln n_{i}}$ is summable
(recall that $\ln \ln n_{i} \sim \ln i$) we get \begin{align*}
\sum _{i_{0} \leqslant i < j \leqslant k}\!\!\!
\mathbb {P}_{x} \bigl[A_{n_{i}} \cap A_{n_{j}}, |\smash {\widehat {S}}_{n_{i}}| \in I_{2}\bigr]
&\leqslant \!\sum _{j = i_{0} + 1}^{k} \!\frac{(2\delta + o(1))}{j \ln j}
\sum _{i=i_{0}}^{j-1} \frac{c}{e^{\ln ^{2} \ln n_{i}}}
\leqslant c_{\delta } \ln \ln k.
\end{align*}

\begin{figure}[ht]
\centering
\begingroup%
  \makeatletter%
  \providecommand\color[2][]{%
    \errmessage{(Inkscape) Color is used for the text in Inkscape, but the package 'color.sty' is not loaded}%
    \renewcommand\color[2][]{}%
  }%
  \providecommand\transparent[1]{%
    \errmessage{(Inkscape) Transparency is used (non-zero) for the text in Inkscape, but the package 'transparent.sty' is not loaded}%
    \renewcommand\transparent[1]{}%
  }%
  \providecommand\rotatebox[2]{#2}%
  \newcommand*\fsize{\dimexpr\f@size pt\relax}%
  \newcommand*\lineheight[1]{\fontsize{\fsize}{#1\fsize}\selectfont}%
  \ifx\svgwidth\undefined%
    \setlength{\unitlength}{381.17416082bp}%
    \ifx\svgscale\undefined%
      \relax%
    \else%
      \setlength{\unitlength}{\unitlength * \real{\svgscale}}%
    \fi%
  \else%
    \setlength{\unitlength}{\svgwidth}%
  \fi%
  \global\let\svgwidth\undefined%
  \global\let\svgscale\undefined%
  \makeatother%
  \begin{picture}(1,0.46651341)%
    \lineheight{1}%
    \setlength\tabcolsep{0pt}%
    \put(0,0){\includegraphics[width=\unitlength,page=1]{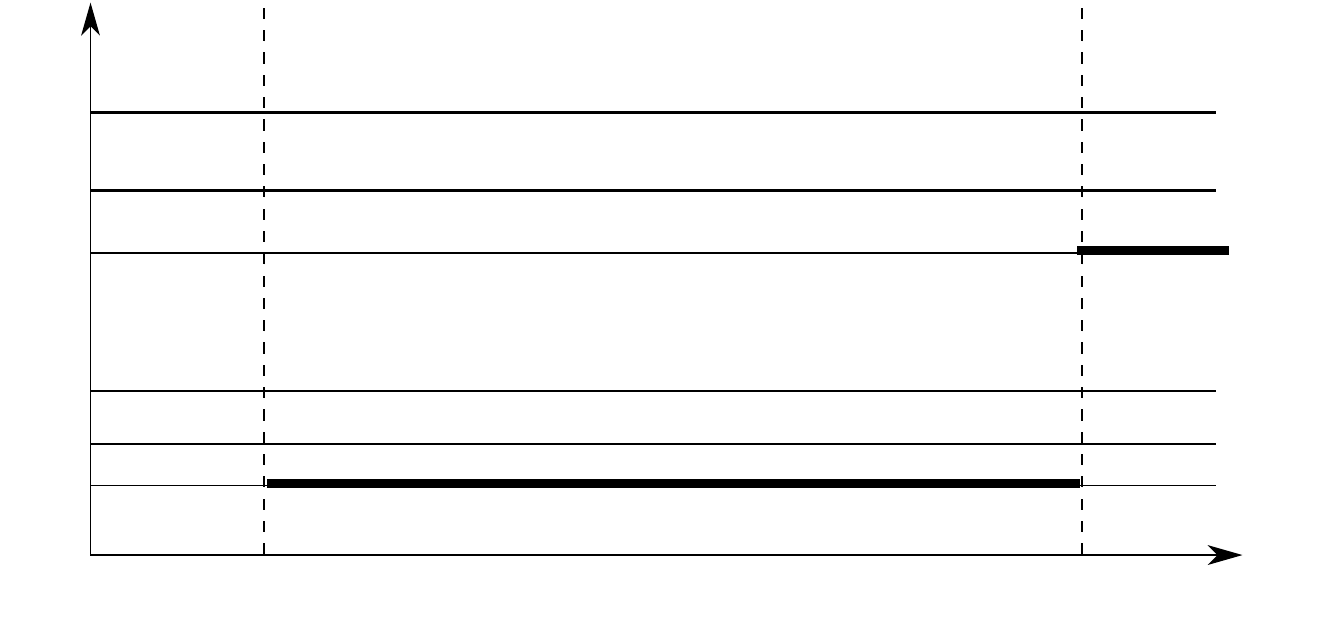}}%
    \put(-0.00391982,0.09582163){\color[rgb]{0,0,0}\makebox(0,0)[lt]{\lineheight{1.25}\smash{\begin{tabular}[t]{l}$m_i$\end{tabular}}}}%
    \put(-0.00391982,0.12481674){\color[rgb]{0,0,0}\makebox(0,0)[lt]{\lineheight{1.25}\smash{\begin{tabular}[t]{l}$l_i$\end{tabular}}}}%
    \put(-0.00391982,0.16171948){\color[rgb]{0,0,0}\makebox(0,0)[lt]{\lineheight{1.25}\smash{\begin{tabular}[t]{l}$u_i$\end{tabular}}}}%
    \put(-0.00391984,0.26908421){\color[rgb]{0,0,0}\makebox(0,0)[lt]{\lineheight{1.25}\smash{\begin{tabular}[t]{l}$m_j$\end{tabular}}}}%
    \put(-0.00391984,0.31237144){\color[rgb]{0,0,0}\makebox(0,0)[lt]{\lineheight{1.25}\smash{\begin{tabular}[t]{l}$l_j$\end{tabular}}}}%
    \put(-0.00391984,0.36746444){\color[rgb]{0,0,0}\makebox(0,0)[lt]{\lineheight{1.25}\smash{\begin{tabular}[t]{l}$u_j$\end{tabular}}}}%
    \put(0.18188396,0.00745985){\color[rgb]{0,0,0}\makebox(0,0)[lt]{\lineheight{1.25}\smash{\begin{tabular}[t]{l}$n_i$\end{tabular}}}}%
    \put(0.81938776,0.00745985){\color[rgb]{0,0,0}\makebox(0,0)[lt]{\lineheight{1.25}\smash{\begin{tabular}[t]{l}$n_j$\end{tabular}}}}%
    \put(0,0){\includegraphics[width=\unitlength,page=2]{KochenStone_Ai_Aj.pdf}}%
    \put(0.31399456,0.00840118){\color[rgb]{0,0,0}\makebox(0,0)[lt]{\lineheight{1.25}\smash{\begin{tabular}[t]{l}$\hT_{ij}$\end{tabular}}}}%
    \put(0,0){\includegraphics[width=\unitlength,page=3]{KochenStone_Ai_Aj.pdf}}%
    \put(-0.00391984,0.43042785){\color[rgb]{0,0,0}\makebox(0,0)[lt]{\lineheight{1.25}\smash{\begin{tabular}[t]{l}$|\hS_n|$\end{tabular}}}}%
    \put(0.91383276,0.00745985){\color[rgb]{0,0,0}\makebox(0,0)[lt]{\lineheight{1.25}\smash{\begin{tabular}[t]{l}$n$\end{tabular}}}}%
  \end{picture}%
\endgroup%
\captionsetup{singlelinecheck=off}
\caption[]{
    On Lemma~\ref{lema:KochenStone_denom_simpler} we formalize the idea that
    on event $A_i \cap A_j$ the most probable scenario is
    \begin{equation*}
        |\hS_{n_i}| \in [m_i, u_i],\quad
        \hT_{ij} \ll n_j,\quad
        |\hS_{n_j}| \in [m_j, u_j]\quad\text{and} \quad
        \{\htau(m_j)=\infty\} \circ \theta_{n_j}.
    \end{equation*}
    }
\label{fig:KochenStone_Ai_Aj}
\end{figure}

\medskip
\noindent
\textbf{Interval $I_{1}$.} The most representative part of
$A_{n_{i}} \cap A_{n_{j}}$ is when
$|\smash {\widehat {S}}_{n_{i}}|\in I_{1}$. The main idea in this case
is to study how large is the stopping time
\begin{equation*}
\smash {\widehat {T}}_{ij} = \htau (l_{j}) \circ \theta _{n_{i}}
\end{equation*}
and then check how large is $|\smash {\widehat {S}}_{n_{j}}|$ (see Figure~\ref{fig:KochenStone_Ai_Aj}).
We use Lemma~\ref{lema:estimate_htau_r} to estimate
$\smash {\widehat {T}}_{ij}$. Fix $\varepsilon \in (0,1)$ and define
$s_{j} = \lfloor \tfrac{n_{j}}{\ln ^{1-\varepsilon } \ln n_{j}}
\rfloor $. For any $z$ with $|z| \in I_{1}$ we have for
$j > i \geqslant i_{0}(x)$ that
\begin{align*}
\mathbb {P}_{z} \left [\htau (l_{j}) > s_{j} \right ] &\leqslant a(z)^{-1}
e^{ -c_{0} \ln ^{1+\varepsilon } \ln n_{j} } \leqslant
\frac{c}{\ln n_{i}} \cdot e^{ -c_{0} \ln ^{1+\varepsilon } \ln n_{j} }
\\
&\leqslant \frac{c}{i \ln ^{2} i} \cdot
\frac{1}{(j \ln ^{2} j)^{c_{0}\ln ^{\varepsilon } \ln n_{j}}} := b_{ij},
\end{align*}
which is summable, i.e.,
$\sum _{i,j=i_{0}}^{\infty } b_{ij} < \infty $. Notice that
\begin{align*}
\mathbb {P}_{x} \bigl[A_{n_{i}} \cap A_{n_{j}}, |
\smash {\widehat {S}}_{n_{i}}| \in I_{1}\bigr] &\leqslant \mathbb {P}_{x}
\bigl[ A_{n_{i}} \cap A_{n_{j}}\ \big|\
 |\smash {\widehat {S}}_{n_{i}}| \in I_{1} \bigr]
\\
&\leqslant \max _{|z| \in I_{1}} \mathbb {P}_{z} \biggl[
\begin{array}{c}
\htau (l_{j}) \leqslant s_{j},\ n_{j} - n_{i} < \htau (m_{i}),
\\
\{\htau (m_{j}) = \infty \} \circ \theta _{n_{j} - n_{i}}
\end{array} \biggr] + b_{ij}.
\end{align*}
Denote the maximum above by $a_{ij}$. To estimate $a_{ij}$ we use Markov
property at time $\htau (l_{j})$, obtaining
\begin{align*}
a_{ij} &= \max _{|z| \in I_{1}} \sum _{t=1}^{s_{j}} \mathbb {P}_{z} [
\htau (l_{j}) = t < \htau (m_{i})] \cdot \mathbb {P}_{l_{j}} \left [
\begin{array}{c}
\htau (m_{i}) > n_{j} - n_{i} - t,
\\
\{\htau (m_{j}) = \infty \} \circ \theta _{n_{j} - n_{i} - t}
\end{array} \right ]
\\
&\leqslant \max _{|z| \in I_{1}} \mathbb {P}_{z} [\htau (l_{j}) <
\htau (m_{i})] \cdot \max _{1 \leqslant t \leqslant s_{j}}
\mathbb {P}_{l_{j}} \left [ \{\htau (m_{j}) = \infty \} \circ \theta _{n_{j}
- n_{i} - t} \right ]
\\
&\leqslant \mathbb {P}_{u_{i}} [\htau (l_{j}) < \htau (m_{i})] \cdot
\max _{1 \leqslant t \leqslant s_{j}} \mathbb {P}_{l_{j}} \left [ \{
\htau (m_{j}) = \infty \} \circ \theta _{n_{j} - n_{i} - t} \right ].
\end{align*}
Using the Markov property once again, together with
$n_{i} + s_{j} \ll n_{j}$, we get
\begin{align*}
\max _{1 \leqslant t \leqslant s_{j}} &\mathbb {P}_{l_{j}} \left [ \{
\htau (m_{j}) = \infty \} \circ \theta _{n_{j} - n_{i} - t} \right ]
\\[-1ex]
&\leqslant \max _{1 \leqslant t \leqslant s_{j}} \left \{ \mathbb {P}_{l_{j}}
\biggl[
\begin{array}{c}
|\smash {\widehat {S}}_{n_{j} - n_{i} - t}| \in [m_{j}, u_{j}],
\\
\{\htau (m_{j}) = \infty \} \circ \theta _{n_{j} - n_{i} - t}
\end{array} \biggr] + \mathbb {P}_{l_{j}} \bigl[ |
\smash {\widehat {S}}_{n_{j} - n_{i} - t}| > u_{j}\bigr] \right \}
\\
&\leqslant \mathbb {P}_{u_{j}}\bigl[\htau (m_{j}) = \infty \bigr] +
\mathbb {P}_{l_{j}} \bigl[ |\smash {\widehat {S}}_{n_{j} (1+o(1))}| > u_{j}
\bigr].
\end{align*}
We can use Lemma~\ref{lema:hS_in_extremes} with
$r = \tfrac{n_{j}}{\ln ^{1-\varepsilon } \ln n_{j}}$ for some
$\varepsilon \in (0, 1/3)$ to estimate
\begin{align*}
\mathbb {P}_{l_{j}} \bigl[ |\smash {\widehat {S}}_{n_{j} (1+o(1))}| > u_{j}
\bigr] &\leqslant \frac{1}{a(l_{j})} \exp \Bigl[ - c \cdot \ln ^{2 - 3
\varepsilon } \ln n_{j} \Bigr] + \exp \Bigl[ - c \cdot
\frac{u_{j}^{2}}{n} \Bigr] + \frac{c}{\sqrt{n_{j}}}
\\
&\leqslant e^{-c \cdot \ln ^{2 - 3 \varepsilon } \ln n_{j}}.
\end{align*}
Putting all these pieces together, we can write
\begin{align}
&\sum _{i_{0} \leqslant i < j \leqslant k}\!\!\!\! \mathbb {P}_{x} [A_{n_{i}}
\cap A_{n_{j}}, |\smash {\widehat {S}}_{n_{i}}| \in I_{1}]
\nonumber
\\
\label{eq:KochenStone_denom_proof_parts}
&%
\hspace{2mm}%
\leqslant \!\!\sum _{i_{0} \leqslant i < j \leqslant k}\!\!\!\!
\mathbb {P}_{u_{i}} [\htau (l_{j}) < \htau (m_{i})] \cdot \mathbb {P}_{u_{j}}[
\htau (m_{j}) = \infty ]\ + \!\!\!\sum _{i_{0} \leqslant i < j
\leqslant k}\!\!\!\!
\frac{%
\mathbb {P}_{u_{i}} [\htau (l_{j}) < \htau (m_{i})]
}{%
e^{c \cdot \ln ^{2 - 3 \varepsilon } \ln n_{j}}}.
\end{align}
We finish the proof by using~\eqref{eq:KochenStone_perturbation} to estimate
\begin{align*}
\frac{j \ln ^{2} j}{j \ln ^{2} j - i \ln ^{2} i} \cdot
\frac{1}{e^{c \cdot \ln ^{2 - 3 \varepsilon } \ln n_{j}}} &\leqslant
\frac{j \ln ^{2} j}{j \ln ^{2} j - (j-1) \ln ^{2} (j-1)} \cdot
\frac{1}{e^{c \cdot \ln ^{2 - 3 \varepsilon } \ln n_{j}}}
\\
&= j (1 + o(1)) \cdot
\frac{1}{e^{c \cdot \ln ^{2 - 3 \varepsilon } \ln n_{j}}} \ll
\frac{1}{j^{2}},
\end{align*}
which implies that the second sum on equation~\eqref{eq:KochenStone_denom_proof_parts}
is bounded by
\begin{equation*}
\!\!\!\sum _{i_{0} \leqslant i < j \leqslant k}\!\!\!\!
\frac{%
\mathbb {P}_{u_{i}} [\htau (l_{j}) < \htau (m_{i})]
}{%
e^{c \cdot \ln ^{2 - 3 \varepsilon } \ln n_{j}}} \leqslant \sum _{i=i_{0}}^{k}
\frac{2\delta + o(1)}{i \ln i} \sum _{j=i+1}^{k} \frac{1}{j^{2}}
\leqslant c_{\delta } \cdot \ln \ln k.\qedhere
\end{equation*}
\end{proof}

We conclude this section with the last missing piece to get Theorem~\ref{teo:speed}.

\begin{proof}[Proof of Lemma~\ref{lema:double_sum_estimate}]
Since $j > i$ we can write for $j = i+s$ that
\begin{equation*}
\frac{\ln j}{j \ln ^{2} j - i \ln ^{2} i} \leqslant
\frac{\ln j}{j \ln ^{2} j - i \ln ^{2} j} = \frac{1}{s \ln (i+s)}.
\end{equation*}
Notice that if we sum only on the range $i \leqslant \ln k$, we obtain
\begin{align*}
\sum _{i= i_{0}}^{\ln k} \frac{1}{i \ln i} \sum _{s=1}^{k-i}
\frac{1}{s \ln (i+s)} &\leqslant \sum _{i= i_{0}}^{\ln k}
\frac{1}{i \ln i} \sum _{s=1}^{k-i} \frac{1}{s \ln s} \
 \leqslant \sum _{i= i_{0}}^{\ln k} \frac{\ln \ln k}{i \ln i}
\\
&\leqslant (\ln \ln k) (\ln \ln \ln k),
\end{align*}
which is $o\bigl(\ln ^{2} \ln k\bigr)$. A second observation is that on
the range $1 \leqslant s \leqslant i$ we have
\begin{equation*}
\sum _{i= i_{0}}^{k} \frac{1}{i \ln i} \sum _{s=1}^{i}
\frac{1}{s \ln (i+s)} \leqslant \sum _{i= i_{0}}^{k}
\frac{1}{i \ln ^{2} i} \sum _{s=1}^{i} \frac{1}{s} \leqslant \sum _{i=
i_{0}}^{k} \frac{1}{i \ln i} \leqslant \ln \ln k.
\end{equation*}
Finally, for $s > i > \ln k$ we can write
$\ln (1 + \tfrac{i}{s}) \geqslant \tfrac{i}{2s}$ and thus
\begin{align*}
\sum _{s=i+1}^{k-i} \frac{1}{s \ln (i + s)} &\leqslant \sum _{s=i+1}^{k-i}
\frac{1}{s \ln s + i/2} \leqslant \sum _{s=i+1}^{k} \frac{1}{s \ln s} (1
+ O(\ln ^{-1} k))
\\
&\leqslant (1 + o(1)) [\ln \ln k - \ln \ln i],
\end{align*}
implying that
\begin{align*}
\sum _{i = i_{0}}^{k} \frac{1}{i \ln i} \sum _{s=1}^{k-i}
\frac{1}{s \ln (i+s)} &\leqslant o\bigl(\ln ^{2} \ln k\bigr) + \sum _{i=
\ln k}^{k} \frac{(1 + o(1))}{i \ln i} [\ln \ln k - \ln \ln i]
\\
&\leqslant (1+o(1)) \left [ \ln ^{2} \ln k - \sum _{i=\ln k}^{k}
\frac{\ln \ln i}{i \ln i} \right ].
\end{align*}
We finish the proof by noticing that we can compare the sum above with
the integral of $\tfrac{\ln \ln x}{x \ln x}$, whose anti-derivative is
$\tfrac{1}{2} \ln ^{2} \ln x$.
\end{proof}

\section{Encounters of two walks}
\label{sec:encounters_of_two_walks}

In this section we use the local CLT derived for the
$\smash {\widehat {S}}$-walk to prove that two independent copies of
$\smash {\widehat {S}}$ will meet infinitely often. The idea is to apply
the second moment method to a sequence of random variables that count the
number of encounters on some well-separated time scales. In order to exemplify
the method in a more classical setting, we first prove the following.
\begin{proposition}
\label{prop:SRW_recurrent}
Simple random walk on $\mathbb {Z}^{2}$ is recurrent.
\end{proposition}

\begin{proof}
We prove that $S_{n}$ is recurrent by considering
\begin{equation*}
N_{k} = \sum _{n=b_{k}}^{b_{k+1}} \ind \{ S_{n}=0 \} \ \text{and} \
A_{k} = \ind \{ N_{k} \geqslant 1\} = \{S_{n}
\text{ visits $0$ during } [b_{k},b_{k+1}]\}
\end{equation*}
where $b_{k} = \lfloor e^{3^{k}} \rfloor $. The goal is to show that
$\mathbb {P}_{0} [A_{k} \text{ i.o.}] = 1$, once we prove that
\begin{equation}
\label{eq:SRW_low_bound_cond_BC}
\mathbb {P}_{0} [ A_{k} \mid \mathcal {F}_{b_{k-1}} ] \geqslant
\delta .
\end{equation}
The result will follow from a conditional Borel-Cantelli argument. Relation
\eqref{eq:SRW_low_bound_cond_BC} is obtained by applying the second moment
method to $N_{k}$. Indeed, we have that
\begin{equation*}
\mathbb {P}_{0} [ S_{n} = 0 \mid \mathcal {F}_{b_{k-1}}] =
\mathbb {P}_{S_{b_{k-1}}} [S_{n - b_{k-1}} = 0].
\end{equation*}
Since the walk starts at the origin, we have
$S_{b_{k-1}} \in B(b_{k-1})$ a.s., and for $n\in [b_{k},b_{k+1}]$ notice
that
\begin{equation*}
n - b_{k-1} = n (1 + O(b_{k-1}^{-2})).
\end{equation*}
Thus, for any $x \in B(b_{k-1})$ and $n$ with right parity we have
\begin{equation*}
\mathbb {P}_{x}\bigl[ S_{n - b_{k-1}}\!\!=0 \bigr] =
\frac{2}{\pi (n - b_{k-1})} \exp \Bigl[-\frac{|x|^{2}}{n - b_{k-1}}
\Bigr] + O(n^{-2}) = \frac{2}{\pi n} \big( 1 + O \big( b_{k-1}^{-1}
\big) \big).
\end{equation*}
Summing on $n \in [b_{k}, b_{k+1}]$,
\begin{equation*}
\mathbb {E}_{0} [N_{k} \mid \mathcal {F}_{b_{k-1}}] = \sum _{n=b_{k}}^{b_{k+1}}
\mathbb {P}_{0} [ S_{n}=0 \mid \mathcal {F}_{b_{k-1}}] \asymp \sum _{n=b_{k}}^{b_{k+1}}
\frac{1}{n} + O(\tfrac{1}{n b_{k-1}}) \asymp \ln
\frac{b_{k+1}}{b_{k}} \asymp 3^{k}.
\end{equation*}
For the second moment, we have for $x \in B(b_{k-1})$, and
$n, m+n \in [b_{k}, b_{k+1}]$ with adequate parities that
\begin{equation*}
\mathbb {P}_{x} \left [ S_{n - b_{k-1}}=0, S_{n+m-b_{k-1}}=0 \right ] =
\mathbb {P}_{x} \left [ S_{n - b_{k-1}}=0 \right ] \mathbb {P}_{0}
\left [ S_{m}=0 \right ] \leqslant \frac{C}{nm}
\end{equation*}
and summing for $n, m+n \in [b_{k}, b_{k+1}]$ we get
\begin{align*}
\sum _{n=b_{k}}^{b_{k+1}} &\sum _{m=1}^{b_{k+1} - n} \mathbb {P}_{0} [S_{n}
= 0, S_{m+n} = 0 \mid \mathcal {F}_{b_{k-1}}] \leqslant \sum _{n=b_{k}}^{b_{k+1}}
\sum _{m=2}^{b_{k+1}} \frac{C}{mn} \leqslant \ln b_{k+1} \cdot \sum _{n=b_{k}}^{b_{k+1}}
\frac{C}{n}
\\
&\leqslant C \ln b_{k+1} (\ln b_{k+1} - \ln b_{k} + 1) \leqslant C
\ln ^{2} b_{k+1} = 9C \cdot 3^{2k}.
\end{align*}
Thus, we have by Paley-Zygmund inequality
\begin{equation*}
\mathbb {P}_{0} [A_{k} \mid \mathcal {F}_{b_{k-1}}] = \mathbb {P}_{0} [N_{k}
\geqslant 1 \mid \mathcal {F}_{b_{k-1}}] \geqslant
\frac{(\mathbb {E}_{0}[N_{k} \mid \mathcal {F}_{b_{k-1}}])^{2}}{\mathbb {E}_{0}[N_{k}^{2} \mid \mathcal {F}_{b_{k-1}}]}
\geqslant \frac{c 3^{2k}}{C 3^{2k} + c 3^{k}} \geqslant c,
\end{equation*}
a positive constant. As a consequence, we can write
\begin{equation*}
\sum _{n\geqslant 1} \mathbb {P}_{0} [A_{2n} \mid \mathcal {F}_{b_{2n -
1}}] = \infty \quad \text{$\mathbb {P}_{0}$-almost surely.}
\end{equation*}
The statement that $\mathbb {P}_{0}[A_{k} \text{ i.o.}] = 1$ follows from
applying a conditional Borel-Cantelli lemma
\cite[Theorem~5.3.2]{durrett2010probability}.
\end{proof}

Now, the idea is to adapt the previous proof to show Theorem
\ref{teo:encounters}.

\begin{proof}[Proof of Theorem~\ref{teo:encounters}]
Consider the random variables
\begin{equation*}
N_{k}' := \sum _{n=b_{k}}^{b_{k+1}} \ind _{\smash {\widehat {S}}^{{1}}_{n}=
\smash {\widehat {S}}^{{2}}_{n}}.
\end{equation*}
We want to prove that
\begin{equation*}
\mathbb {P}\left [ \smash {\widehat {S}}_{n}^{1} =
\smash {\widehat {S}}_{n}^{2} \middle | \mathcal {F}_{b_{k-1}}
\right ] \geqslant \frac{c}{n} \quad \text{and} \quad \mathbb {P}
\left [ \smash {\widehat {S}}_{n}^{1}=\smash {\widehat {S}}_{n}^{2} ,
\smash {\widehat {S}}_{n+m}^{1} = \smash {\widehat {S}}_{n+m}^{2}
\middle | \mathcal {F}_{b_{k-1}} \right ] \leqslant
\frac{C}{m\cdot n}
\end{equation*}
for all $b_{k} \leqslant n \leqslant b_{k+1}$ and
$1 \leqslant m \leqslant b_{k+1} - n$, if $k$ is large enough. Considering
only one walk, we have by Corollary~\ref{coro:hS_extremes_simple} that
if we denote $r_{k} = \sqrt{b_{k}\ln b_{k}}$ then
\begin{equation*}
\mathbb {P}_{x} \bigl[ |\smash {\widehat {S}}_{b_{k}}| > r_{k}\bigr] =
\mathbb {P}_{x} \left [ |\smash {\widehat {S}}_{b_{k}}| > \sqrt{b_{k}
\ln b_{k}} \right ] = O \big( \tfrac{1}{\sqrt{b_{k}}} \big),
\end{equation*}
which is summable, so Borel-Cantelli lemma implies
$|\smash {\widehat {S}}_{b_{k}}| \leqslant r_{k}$ eventually.

\medskip
\noindent
\textbf{Lower bound.} For the first bound we just apply Proposition~\ref{prop:lclt}.
For every $z_{i} \in B(r_{k-1})$ and $n\in [b_{k},b_{k+1}]$ and
$y \in \mathbb {Z}^{2}$ with
$\sqrt{b_{k}} \leqslant |y| \leqslant 2\sqrt{n}$ we have
$\ln |y| \asymp \ln n$. In this case, we can write
\begin{align*}
\mathbb {P}_{z_{1}, z_{2}} \bigl[\smash {\widehat {S}}^{1}_{n} =
\smash {\widehat {S}}^{2}_{n}\bigr] &\geqslant \sum _{\sqrt{b_{k}}
\leqslant |y| \leqslant 2\sqrt{n}} \mathbb {P}_{z_{1}} \bigl[
\smash {\widehat {S}}^{1}_{n} = y\bigr] \mathbb {P}_{z_{2}} \bigl[
\smash {\widehat {S}}^{2}_{n} = y\bigr]
\\
&\geqslant \frac{c}{n^{2}} \cdot \# \bigl\{y; \sqrt{b_{k}} \leqslant |y|
\leqslant 2 \sqrt{n} \bigr\} \geqslant \frac{c}{n}.
\end{align*}
Apply the Markov property at time $b_{k-1}$. On event
$|\smash {\widehat {S}}_{b_{k-1}}^{i}| \leqslant r_{k-1}$ for
$i=1,2$ we have
\begin{equation}
\label{eq:low_bound_encounters}
\mathbb {P}_{x_{1}, x_{2}} \big[ \smash {\widehat {S}}^{1}_{n} =
\smash {\widehat {S}}^{2}_{n} \mid \mathcal {F}_{b_{k-1}} \big] =
\mathbb {P}_{\smash {\widehat {S}}^{1}_{b_{k-1}},
\smash {\widehat {S}}^{2}_{b_{k-1}}} \big[ \smash {\widehat {S}}^{1}_{n-b_{k-1}}
= \smash {\widehat {S}}^{2}_{n-b_{k-1}} \big] \geqslant \frac{c}{n},
\end{equation}
since $n \sim n - b_{k-1}$.

\medskip
\noindent
\textbf{Upper bound.} Notice that if $z_{i} \in B(r_{k-1})$ and
$n \in [b_{k}, b_{k+1}]$ then we have for small $\varepsilon > 0$
\begin{equation*}
|z_{i}| \leqslant \sqrt{b_{k-1} \ln b_{k-1}} \ll b_{k-1}^{1/2 + 3
\varepsilon } = b_{k}^{1/6 + \varepsilon } \leqslant n^{1/3}.
\end{equation*}
Using Proposition~\ref{prop:lclt}, we can write
\begin{align*}
\sum _{0 < |y| \leqslant \sqrt{n}} \mathbb {P}_{z_{1}, z_{2}} \bigl[
\smash {\widehat {S}}_{n}^{1} = \smash {\widehat {S}}_{n}^{2} = y
\bigr] &\leqslant \sum _{0 < |y| \leqslant \sqrt{n}} \mathbb {P}_{z_{1}}
\bigl[\smash {\widehat {S}}_{n} = y\bigr] \mathbb {P}_{z_{2}} \bigl[
\smash {\widehat {S}}_{n} = y\bigr]
\\
&\leqslant \# \bigl\{y;\; |y| \leqslant \sqrt{n} \bigr\} \cdot
\frac{C}{n^{2}}
\\
&\leqslant \frac{C}{n}.
\end{align*}
Also, if we choose $s = (n \ln n^{1/2})^{1/2}$ in Corollary~\ref{coro:hS_extremes_simple},
we have
\begin{align*}
\sum _{s < |y|} \mathbb {P}_{z_{1}, z_{2}} \bigl[
\smash {\widehat {S}}_{n}^{1} = \smash {\widehat {S}}_{n}^{2} = y
\bigr] &\leqslant \mathbb {P}_{z_{1}} \bigl[|\smash {\widehat {S}}_{n}|
> s\bigr] \cdot \mathbb {P}_{z_{2}} \bigl[|\smash {\widehat {S}}_{n}| >
s\bigr]
\\
&\leqslant \left ( c e^{- \frac{s^{2}}{n}} + c n^{-1/2} \right )^{2}
\\
&\leqslant \frac{c}{n}.
\end{align*}
We decompose the remaining range $\sqrt{n} < |y| \leqslant s$ into a collection
of disjoint intervals $I_{k} = (e^{k-1} \sqrt{n}, e^{k} \sqrt{n}]$. Notice
that to cover the range above we only need to consider
$1 \leqslant k \leqslant \lceil \tfrac{1}{2} \ln \ln n \rceil $. If
$|y| \in I_{k}$, we have for $z \in \partial B(n^{1/3})$ that
\begin{align*}
\mathbb {P}_{z} \bigl[\smash {\widehat {S}}_{n} = y\bigr] &\leqslant
\frac{a(y)}{a(z)} \cdot \mathbb {P}_{z} \bigl[S_{n} = y\bigr]
\\
&\leqslant \frac{k + \frac{1}{2} \ln n}{\frac{1}{3} \ln n} \cdot
\left [ \frac{2}{\pi n} e^{- \frac{|y|^{2}}{n}} + |y|^{-2} O(n^{-1})
\right ]
\\
&\leqslant \frac{C}{n} e^{- e^{2k-2}} + \frac{C}{n^{2}} e^{-2k},
\end{align*}
where the second inequality uses the error term from the local CLT given
in \cite[Theorem~1.2.1]{lawler2013intersections} and the last inequality
uses that $k = O(\ln \ln n)$.

For any $z_{i} \in B(r_{k-1})$, we can use once again the method of decomposing
with respect to the first time the walk hits $\partial B(n^{1/3})$. The
bound in equation~\eqref{eq:htau_decomposition} gives
\begin{equation*}
\mathbb {P}_{z_{i}} \bigl[\smash {\widehat {S}}_{n} = y\bigr]
\leqslant e^{-c n^{\delta }} + \max _{
\substack{
z \in \partial B(n^{1/3}) \\
1 \leqslant j \leqslant n^{2/3 + \delta }
}} \mathbb {P}_{z} \bigl[\smash {\widehat {S}}_{n-j} = y\bigr]
\leqslant \frac{C}{n} e^{- e^{2k-2}} + \frac{C}{n^{2}} e^{-2k},
\end{equation*}
since $e^{-c n^{\delta }} \ll n^{-l}$ for every $l \in \mathbb {N}$ and
$n-j \sim n$. Summing for $|y| \in I_{k}$ we get
\begin{align*}
\sum _{|y| \in I_{k}} \mathbb {P}_{z_{1}, z_{2}} \bigl[
\smash {\widehat {S}}_{n}^{1} = \smash {\widehat {S}}_{n}^{2} = y
\bigr] &\leqslant \# \{y;\; |y| \in I_{k} \} \cdot \left (\frac{C}{n} e^{-
e^{2k-2}} + \frac{C}{n^{2}} e^{-2k}\right )^{2}
\\
&\leqslant C (n e^{2k}) \left ( \frac{1}{n^{2}} e^{- 2e^{2k-2}} +
\frac{1}{n^{3}} e^{- e^{2k-2} - 2k} + \frac{1}{n^{4}} e^{-4k} \right )
\\
&\leqslant C \left ( \frac{1}{n} e^{- 2e^{2k-2} + 2k} +
\frac{1}{n^{2}} e^{- e^{2k-2}} + \frac{1}{n^{3}} e^{-2k} \right ).
\end{align*}
Each of the three terms on the right hand side is summable in $k$. Hence,
\begin{equation*}
\sum _{k=1}^{\lceil \frac{1}{2} \ln \ln n \rceil } \sum _{|y| \in I_{k}}
\mathbb {P}_{z_{1}, z_{2}} \bigl[\smash {\widehat {S}}_{n}^{1} =
\smash {\widehat {S}}_{n}^{2} = y\bigr] \leqslant \frac{C}{n},
\end{equation*}
concluding that
$\mathbb {P}_{z_{1}, z_{2}} \bigl[\smash {\widehat {S}}_{n}^{1} =
\smash {\widehat {S}}_{n}^{2}\bigr] \leqslant \frac{C}{n}$. Using Markov
property and the bound above, we can write that on event
$|\smash {\widehat {S}}_{b_{k-1}}^{i}| \in B(r_{k-1})$ for $i =1,2$ we
have
\begin{align*}
\mathbb {P}_{x_{1}, x_{2}} &\bigl[\smash {\widehat {S}}_{n}^{1} =
\smash {\widehat {S}}_{n}^{2},\, \smash {\widehat {S}}_{n+m}^{1} =
\smash {\widehat {S}}_{n+m}^{2} \mid \mathcal {F}_{b_{k-1}}\bigr]
\\
&= \mathbb {P}_{\smash {\widehat {S}}_{b_{k-1}}^{1},
\smash {\widehat {S}}_{b_{k-1}}^{2}} \bigl[ \smash {\widehat {S}}_{n-b_{k-1}}^{1}
= \smash {\widehat {S}}_{n-b_{k-1}}^{2},\, \smash {\widehat {S}}_{n+m-b_{k-1}}^{1}
= \smash {\widehat {S}}_{n+m-b_{k-1}}^{2} \bigr]
\\
&\leqslant \frac{C}{n} \cdot \mathbb {P}_{\smash {\widehat {S}}_{b_{k-1}}^{1},
\smash {\widehat {S}}_{b_{k-1}}^{2}} \bigl[ \smash {\widehat {S}}_{n+m-b_{k-1}}^{1}
= \smash {\widehat {S}}_{n+m-b_{k-1}}^{2} \mid \smash {\widehat {S}}_{n-b_{k-1}}^{1}
= \smash {\widehat {S}}_{n-b_{k-1}}^{2} \bigr].
\end{align*}
Finally, we have from Proposition~\ref{prop:lclt} and the independence
of walks $\smash {\widehat {S}}^{i}$ that
\begin{equation}
\label{eq:up_bound_encounters}
\mathbb {P}_{x_{1}, x_{2}} \bigl[\smash {\widehat {S}}_{n}^{1} =
\smash {\widehat {S}}_{n}^{2},\, \smash {\widehat {S}}_{n+m}^{1} =
\smash {\widehat {S}}_{n+m}^{2} \mid \mathcal {F}_{b_{k-1}}\bigr]
\leqslant \frac{C}{n} \cdot \max _{x, y \neq 0} \mathbb {P}_{x}
\bigl[\smash {\widehat {S}}_{m}^{1} = y\bigr] \leqslant \frac{C}{nm}.
\end{equation}

\noindent
\textbf{Encounters.} Having estimates~\eqref{eq:low_bound_encounters} and~\eqref{eq:up_bound_encounters},
we proceed with the second moment method. Let us consider the events
\begin{equation*}
V_{k} = \bigcap _{i=1, 2} \{ |\smash {\widehat {S}}_{b_{k-1}}^{i}|
\leqslant r_{k-1} \} \quad \text{and} \quad V_{k}^{+} = \bigcap _{l
\geqslant k} V_{l} \quad \text{and} \quad A'_{k} = \{ N'_{k}
\geqslant 1 \}.
\end{equation*}
Just like in Proposition~\ref{prop:SRW_recurrent}, we have that equations~\eqref{eq:low_bound_encounters}
and \eqref{eq:up_bound_encounters} give for $k \geqslant k_{0}$ that
\begin{equation*}
\ind _{V_{k_{0}}^{+}} \mathbb {E}_{x_{1}, x_{2}} [N'_{k} \mid
\mathcal {F}_{b_{k-1}}] \asymp \ind _{V_{k_{0}}^{+}} 3^{k} \
\text{and} \
\ind _{V_{k_{0}}^{+}} \mathbb {E}_{x_{1}, x_{2}} [(N'_{k})^{2} \mid
\mathcal {F}_{b_{k-1}}] \leqslant \ind _{V_{k_{0}}^{+}} C 3^{2k}.
\end{equation*}
and Paley-Zygmund inequality implies
\begin{equation*}
\ind _{V_{k_{0}}^{+}} \cdot \mathbb {P}_{x_{1}, x_{2}} [A'_{k} \mid
\mathcal {F}_{b_{k-1}}] \geqslant \ind _{V_{k_{0}}^{+}} \cdot
\frac{
(\mathbb {E}_{x_{1}, x_{2}} [N'_{k} \mid \mathcal {F}_{b_{k-1}}])^{2}
}{
\mathbb {E}_{x_{1}, x_{2}} [(N'_{k})^{2} \mid \mathcal {F}_{b_{k-1}}]
} \geqslant \ind _{V_{k_{0}}^{+}} \cdot c,
\end{equation*}
where $c$ is a positive constant. This implies that on
$V_{k_{0}}^{+}$ we have
\begin{equation}
\label{eq:infty_sum_cond_BC}
\sum _{k\geqslant 1} \mathbb {P}_{x_{1}, x_{2}} [A'_{2k} \mid
\mathcal {F}_{2k-1}] = \infty .
\end{equation}
Since
$\mathbb {P}_{x_{1}, x_{2}} [ V_{k},\ \text{eventually}] = \mathbb {P}_{x_{1},
x_{2}} [ \cup _{k\geqslant 1} V_{k}^{+}] = 1$, we conclude that~\eqref{eq:infty_sum_cond_BC}
holds $\mathbb {P}_{x_{1}, x_{2}}$-almost surely and finish the proof with
the conditional Borel-Cantelli lemma~\cite[Theorem~5.3.2]{durrett2010probability}.
\end{proof}

\providecommand{\bysame}{\leavevmode\hbox to3em{\hrulefill}\thinspace}

\ACKNO{
S.P.\ thanks NYU-Shanghai for support and hospitality.
The work of S.P.\ was partially supported by CNPq
(301605/2015--7).
The work of D.U.\ was supported by FAPESP
(2017/16294--4).
}

\end{document}